\documentclass[preprint,12pt]{elsarticle}

\usepackage{amssymb}
\usepackage{amsmath}
\usepackage{amsthm}

\newtheorem{theorem}{Theorem}[section] 

\newtheorem{lemma}[theorem]{Lemma}

\newtheorem{definition}[theorem]{Definition} 

\newtheorem{proposition}[theorem]{Proposition}

\numberwithin{equation}{section}    

\usepackage{hyperref} 

\hypersetup{
    colorlinks=true,      
    linkcolor=blue,       
    citecolor=red,        
    urlcolor=magenta,     
    filecolor=cyan,       
    allcolors=blue,       
    pdfborder={0 0 0}     
}

\journal{}

\begin{document}

\begin{frontmatter}



\title{Normalized solutions of $L^2$ supercritical NLS equations in exterior domains with inhomogeneous nonlinearities} 

\author[label1]{Xiaojun Chang}
\ead{changxj100@nenu.edu.cn }
\author{Cong-Mei Li\corref{cor1}\fnref{label2,label3}}
\ead{ licongmei@amss.ac.cn }
\cortext[cor1]{Corresponding Author}

\affiliation[label1]{
    organization={School of Mathematics and Statistics $\&$ Center for Mathematics and Interdisciplinary Sciences},
    addressline={Northeast Normal University},
    city={Changchun},
    postcode={130024},
    state={Jilin},
    country={PR China}
}

\affiliation[label2]{organization={Academy of Mathematics and Systems Science},
    addressline={Chinese Academy of Sciences},
    city={Beijing},
    postcode={100190},
    country={PR China}}
\affiliation[label3]{organization={University of Chinese Academy of Sciences},
    city={Beijing},
    postcode={100049}, 
    country={PR China}}



\begin{abstract}
This paper establishes the existence of normalized mountain pass solutions to the $L^2$-supercritical nonlinear Schrödinger equation with inhomogeneous nonlinearity $|x|^{-\alpha}|u|^{p-2}u$ in exterior domains. In contrast, for the autonomous case ($\alpha=0$), Appolloni \& Molle (2025) and Zhang \& Zhang (2022) showed that potential mountain pass solutions share the same energy levels as in $\mathbb{R}^N$, causing non-existence due to energy leakage to infinity. This work demonstrates that the physically motivated decaying term $|x|^{-\alpha}$ breaks the scaling symmetry inherent in the autonomous case. Such breaking energetically separates the exterior domain problem from the whole space one and thereby prevents energy leakage. Using a novel min-max argument that combines monotonicity trick, Morse index estimates, and blow-up analysis, we prove the existence of a positive mountain pass solution for sufficiently small mass, revealing a new phenomenon of non-autonomous nonlinearities in non-compact domains.

\end{abstract}




\end{frontmatter}


\section{Introduction} \label{sec1}	
In this paper, we study the existence of solutions $(u,\lambda)\in H_0^1(\Omega)\times \mathbb{R}$ to the problem
\begin{equation}\label{1.1}
\begin{cases} 
   -\Delta u + \lambda u = |x|^{-\alpha}|u|^{p-2} u \quad &\text{in } \Omega, \\
   u = 0 \quad&\text{on } \partial\Omega,  
\end{cases}
\end{equation}
under the mass constraint
\begin{equation}\label{1.2}
    \int_{\Omega} u^2 \, dx = \mu>0,
\end{equation}
where \(N\geq3,\) \( \Omega\subset\mathbb{R}^N\) is an exterior domain such that $0\notin \overline\Omega$ (i.e., $\mathbb{R}^N\backslash \Omega$ is compact and there exists $r_0>0$ such that $|x|>r_0$ for all $x\in\Omega$), $\alpha>0$ is fixed and $p\in(2 + 4/N , 2^* )$ with $2^*=2N/(N-2)$. Here $\lambda$ appears as a Lagrange multiplier associated with the mass
constraint $($\ref{1.2}$)$.

A solution with prescribed mass is often referred to as a \textit{normalized solution}, the solutions of equations $($\ref{1.1}$)$-$($\ref{1.2}$)$ correspond to critical points of the energy functional $E:H_0^1(\Omega)\to \mathbb{R}$ defined by 
\begin{equation*}
    E(u):=\frac{1}{2}\int_\Omega
|\nabla u|^2 dx-\frac{1}{p}\int_{\Omega}|x|^{-\alpha}|u|^pdx\end{equation*}
constrained to the $L^2$-sphere
\begin{equation*}
    S_{\mu} := \left\{ u \in H_0^1(\Omega) : \int_{\Omega} u^2 \, dx = \mu \right\}.
\end{equation*}
Equation $($\ref{1.1}$)$ arises naturally in physical contexts such as nonlinear optics, quantum mechanics, and astrophysics, where it describes stationary waves of the nonlinear {S}chr{\"o}dinger equation
\begin{equation*}
\begin{cases}
i \frac{\partial \psi}{\partial t} = -\Delta \psi  + |x|^{-\alpha} |\psi|^{p-2} \psi \quad &\text{in } \Omega \times (0, +\infty),\\
\psi(x,t) = 0 \quad &\text{on } \partial \Omega \times (0, +\infty)
\end{cases}
\end{equation*}
with prescribed $L^2$-norm. Standing waves of the form $\psi(x,t)=e^{i\lambda t}u(x)$ thus motivate the study of normalized solutions. We refer to \cite{berestycki1983nonlinear,cazenave1982orbital} for further physical background.

In the case of $\alpha = 0,$ normalized solutions of nonlinear {S}chr{\"o}dinger equations have been extensively studied. In $\mathbb{R}^N,$ results for mass supercritical nonlinearities appear in \cite{bartsch2021normalized, bartschsoave2017jfa, JeanjeanLu2020}, while combined nonlinearities are treated in \cite{Soavejfa2020, Soavejde2020}. Studies on metric graphs include \cite{Borthwick_2023, changcompact2024} for mass supercritical cases and \cite{PierottiSoave2022} for combined nonlinearities. Systems are considered in \cite{bartsch2016normalized, Ikoma, Jeanjean2024NoDEA, Mederski}. Classical results \cite{berestycki1983nonlinear,jeanjean1997existence,Kwong1989UniquenessOP} guarantee that for any $p\in (2,2^*)$ and $\mu>0,$ there exists a unique radial positive solution for some $\lambda>0$ in $\mathbb{R}^N$. In bounded domains, the existence of normalized solutions is also well studied, for instance, in the unit ball, \cite{noris2015existence} showed that a normalized solution exists for any $\mu>0$ in the subcritical case, and only for small $\mu$ in the supercritical case. Further results on general bounded domains can be found in \cite{pierotti2017normalized, pierotti2025normalized}.

For exterior domains, however, fewer results are available. Notably, \cite{zhang2022normalized} and \cite{appolloni2025normalizedschrodingerequationsmasssupercritical} studied mass subcritical and mass supercritical cases, respectively. Crucially, both works show that when $\alpha=0,$ equation (\ref{1.1})–(\ref{1.2}) does not admit least energy solutions in the mass subcritical regime, nor mountain pass solutions in the mass supercritical regime. When $\alpha = 0,$ the nonlinearity is homogeneous, meaning its strength does not vary with spatial points. This causes the least energy level in \cite{zhang2022normalized} and the mountain pass energy level in \cite{appolloni2025normalizedschrodingerequationsmasssupercritical} coincide with those in $\mathbb{R}^N$. Consequently, any bounded sequence attempting to attain the least energy level or the mountain pass level in an exterior domain will eventually converge (up to translation) to a solution on the whole space rather than $\Omega$. The exterior domain cannot trap the solution and energy leaks to infinity. To overcome the issue of energy level non-attainability, \cite{appolloni2025normalizedschrodingerequationsmasssupercritical,zhang2022normalized} discovered higher energy linking structure by constructing a topologically more complex constraint set to eliminate the effect of energy leakage. 

In contrast, when $\alpha>0,$ we observe a new phenomenon for $\alpha \neq 0$. The physically motivated decaying term $|x|^{-\alpha}$ breaks the scaling symmetry inherent in the autonomous case. Physically, the interaction strength decreases with distance from the origin, creating a natural potential well that traps the solution within the exterior domain. This causes the energy levels in exterior domains to separate strictly from those in the whole space, thereby preventing the energy from leaking to infinity. As a result, we prove that for any sufficiently small mass $\mu>0,$ equations $($\ref{1.1}$)$-$($\ref{1.2}$)$ admit a mountain pass solution as stated in Theorem \ref{main result}.  Such decaying nonlinearities model physically realistic scenarios including non-uniform nonlinear responses, long-range interactions (e.g., Coulomb or dipole potentials), or spatially decaying coupling strengths \cite{Farah2016, Malomed2008}. Regarding the mathematical analysis of equation $($\ref{1.1}$),$ notable contributions have been made by \cite{Alex2021, Bouard2005, CAMPOS2021112118, Gou2023, Liu2006, LUO2022109489}, which extensively study the existence and dynamical behavior of normalized solutions across the whole space $\mathbb{R}^N$. However, to the best of our knowledge, studies in exterior domains remain relatively scarce.


In this paper, we consider the mass supercritical case with $\alpha > 0$. It is straightforward to observe that for any mass $\mu > 0,$ the functional $E$ restricted to $S_\mu$ admits a mountain pass geometry. However, as a key preliminary step toward obtaining a critical point, the existence of a bounded Palais-Smale sequence at the corresponding mountain pass level cannot be directly established. The authors of \cite{borthwick2024bounded} proposed a general abstract framework by building upon the monotonicity trick method developed in \cite{Jeanjean1999,Struwe1988} and the approximate Morse theory introduced in \cite{FangGhoussoub1992,FangGhoussoub1994}. This framework can typically generate a bounded Palais-Smale sequence carrying Morse type information for constrained functionals exhibiting mountain pass geometry. It is particularly useful in situations where Pohozaev type identities either do not exist or fail to provide useful information. Therefore, motivated by \cite{Borthwick_2023,borthwick2024bounded}, we first consider the family of functionals $E_\rho : H_0^1(\Omega)\to \mathbb{R}$ given by 
\begin{equation*}
    E_\rho(u):=\frac{1}{2}\int_\Omega
|\nabla u|^2 dx-\frac{\rho}{p}\int_{\Omega}|x|^{-\alpha}|u|^pdx,\quad \forall u\in H_0^1(\Omega),\forall\rho\in[\frac{1}{2},1].\end{equation*}  

We shall prove that $E_\rho$ possesses a uniform mountain pass structure for all $\rho\in[\frac{1}{2},1]$ (see Theorem \ref{c_rho}), and the mountain pass level denoted by $c_\rho(\mu).$ Instead of attacking directly equation $($\ref{1.1}$)$ under the constraint $($\ref{1.2}$),$ we consider equation $($\ref{1.1}$)$ under a closed ball constraint. This idea of first considering a closed ball constraint is reminiscent of the works \cite{Bieganowski2021,Jeanjean2024NoDEA,Mederski}. However, the strategy of the proofs in \cite{Bieganowski2021,Jeanjean2024NoDEA,Mederski} and in the present paper is essentially distinct. Firstly, we introduce this definition,

\begin{definition}
    For $\mu>0, \rho\in[\frac{1}{2},1],$ a solution $u\in H_0^1(\Omega)$ to 
    \begin{equation}\tag{1.3-$\rho$}\label{1.1rho}
\begin{cases} 
   -\Delta u + \lambda u = \rho{|x|^{-\alpha}}|u|^{p-2} u \quad &\text{in } \Omega, \\
   u = 0 &\text{in } \partial\Omega  
\end{cases}
\end{equation}
is called $(\mu,\rho)$ sub-mass mountain pass solution$,$ if $u\in S_\mu\cup S_\mu^-$ solves $($\ref{1.1rho}$)$ for some $\lambda\in\mathbb{R},$ and $E_\rho(u)=c_\rho(\mu),$ where $S_\mu^-:=\{u\in H_0^1(\Omega):\int_{\Omega}|u|^2 dx< \mu\}.$
\end{definition}

 We shall prove that for any given $\mu>0, \rho\in[\frac{1}{2},1],$ the mountain pass level $c_\rho(\mu)$ is attained by some positive $(\mu,\rho)$ sub-mass solution $u(\mu,\rho)\in H_0^1(\Omega),$ and corresponds to some $\lambda(\mu,\rho)\geq 0.$

\begin{theorem}\label{submass}
    Let $\alpha,\mu>0,$ $N\geq3,$ $2+4/N<p<2^*,$ $\rho\in[\frac{1}{2},1],$ and $\Omega\subset\mathbb{R}^N$ be an exterior domain with $0\notin\overline\Omega$. Then for any $\mu>0,$ $\rho\in[\frac{1}{2},1],$ equation $($\ref{1.1rho}$)$ has a positive $(\mu,\rho)$ sub-mass solution $u(\mu,\rho)$ for some $\lambda(\mu,\rho)\geq 0.$
\end{theorem}

 To deal with the lack of compactness, we only need to rule out the case $\lambda(\mu,\rho)=0$. This might typically involve Liouville type results, such as those in \cite{Ikoma,jeanjean2024JMPA}. However, to our knowledge, existing Liouville-type theorems are not applicable to our problem. Therefore, we develop a new approach by studying the asymptotic behavior of sub-mass solutions (or solutions under a closed ball constraint) to analyze compactness. And inspired by \cite{HE2023375}, we shall obtain the following asymptotic behavior for the $(\mu,\rho)$ sub-mass mountain pass solution $u(\mu,\rho)$ in the sense of $C^2_{loc}(\mathbb{R}^N)$ and $H^1(\mathbb{R}^N)$.


\begin{theorem}\label{asymtotic}
    Under the assumptions of Theorem \ref{submass}. For any $\mu>0,$ keep $\rho$ constant, let $(u(\mu,\rho),\lambda(\mu,\rho))$ be given by Theorem \ref{submass}$,$ then $\lambda(\mu,\rho)\rightarrow +\infty$ as $\mu\rightarrow 0,$ and there is a $P\in\Omega,$ define   
\[v(\mu,\rho):=\lambda(\mu,\rho)^{-\frac{1}{p-2}} u(\mu,\rho)(\frac{x}{\sqrt{\lambda(\mu,\rho)}}  + P), \quad\text{ for }\  x \in \tilde{\Omega}_{\mu,\rho} := \frac{\Omega - P}{\sqrt{\lambda(\mu,\rho)}},\]
then $v(\mu,\rho)\to V_\rho(x)$ in $C_{loc}^{2}(\mathbb{R}^N)$ and $H^1(\mathbb{R}^N)$ when $\mu\to 0,$
where \(V_\rho\) is the unique positive solution of
\[\left\{\begin{array}[]{ll}-\Delta V_\rho+V_\rho=\rho\Lambda_0V_\rho^{p-1}&\quad\text{ in }\mathbb{R}^{N}, \\ V_\rho(0)=\max\limits_{x\in\mathbb{R}^N}V_\rho,&\\ V_\rho(x)\to 0~{}~{}\text{as}~{}~{}|x|\to+\infty,&\end{array}\right.\]
and $\Lambda_0=|P|^{-\alpha}>0.$
\end{theorem}

  In other words, Theorem \ref{asymtotic} provides an a priori estimate of the Lagrange multiplier $\lambda$ in the setting of the mountain pass. Hence, for small enough $\mu,$ we necessarily have $\lambda>0,$ at this point we shall prove that the nonlinearity $|x|^{-\alpha}|u|^{p-2}u$ ensures compactness. And we have the following result: 
\begin{theorem}\label{main result}
Let $\alpha,\mu>0,$ $N\geq3,$ $2+4/N<p<2^*,$ and $\Omega\subset\mathbb{R}^N$ be an exterior domain with $0\notin\overline\Omega$. Then there is a $\mu_0>0$ such that for any $\mu\in(0,\mu_0),$ equations $($\ref{1.1}$)$-$($\ref{1.2}$)$ has a positive solution for some $\lambda> 0.$  
\end{theorem}

This paper is structured as follows. In Section \ref{sec2}, we review key abstract results from \cite{Borthwick_2023,borthwick2024bounded} and investigate several of its consequences. In particular, we prove that the approximate Morse information of the bounded Palais-Smale sequence generated by Theorem \ref{monotonous} can be used to derive an upper bound for the Morse index of its weak limit. In Section \ref{sec3}, we prove that the energy functional $E_\rho$ possesses a mountain pass structure for almost every $\rho\in[\frac{1}{2},1]$. Section \ref{sec4} is devoted to proving that for any $\mu>0$ and almost every $\rho\in[\frac{1}{2},1],$ the functional $E_\rho$ possesses a positive $(\mu,\rho)$ sub-mass mountain pass solution of equation $($\ref{1.1rho}$)$. In Section \ref{sec5}, we present blow-up analysis results and complete the proofs of Theorem \ref{submass}, Theorem \ref{asymtotic}, and Theorem \ref{main result}.

\section{Preliminaries}\label{sec2}
Let $(E,\langle\cdot,\cdot\rangle)$ and $(H,(\cdot,\cdot))$ be two \textit{infinite-dimensional} Hilbert spaces and assume that $E\hookrightarrow H\hookrightarrow E^{\prime},$ with continuous injections. For simplicity, we assume that the continuous injection $E\hookrightarrow H$ has norm at most 1 and identify $E$ with its image in $H$. Set
\[
\begin{cases}
\|u\|^{2}=\langle u,u\rangle,  \\
|u|^{2}=(u,u),
\end{cases}\quad u\in E,
\]
and we define for $\mu>0$:
\[
S_{\mu}=\{u\in E \mid |u|^{2}=\mu\}.
\]
 Clearly, $S_{\mu}$ is a smooth submanifold of $E$ of codimension 1. Its tangent space at a given point $u\in S_{\mu}$ can be considered as the closed codimension 1 subspace of $E$ given by:
\[
T_{u}S_{\mu}=\{v\in E \mid (u,v)=0\}.
\]

In the following definition, we denote by $\|\cdot\|_{\mathcal{L}(E,\mathbb{R})}$ and $\|\cdot\|_{\mathcal{L}(E,\mathcal{L}(E,\mathbb{R}))},$ respectively, the operator norm of $\mathcal{L}(E,\mathbb{R})$ and of $\mathcal{L}(E,\mathcal{L}(E,\mathbb{R}))$.

\begin{definition}
Let $\phi:E\rightarrow\mathbb{R}$ be a $C^{2}$-functional on $E$ and $\xi\in\left(0,1\right]$. We say that $\phi^{\prime}$ and $\phi^{\prime\prime}$ are $\xi$-H\"older continuous on bounded sets if for any $R>0$ one can find $M=M(R)>0$ such that, for any $u_{1},u_{2}\in B(0,R):$
\begin{align*}
&\left\|\phi^{\prime}(u_{1})-\phi^{\prime}(u_{2})\right\|_{\mathcal{L}(E,\mathbb{R})} \leq M \|u_{1}-u_{2}\|^{\xi},\\
\ &\left\|\phi^{\prime\prime}(u_{1})-\phi^{\prime\prime}(u_{2})\right\|_{\mathcal{L}(E,\mathcal{L}(E,\mathbb{R}))} \leq M\|u_{1}-u_{2}\|^{\xi}.
\end{align*}
\end{definition}
\begin{definition}
Let $\phi$ be a $C^{2}$-functional on $E$. For any $u\in E,$ we define the continuous bilinear map:
\[
D^{2}\phi(u)=\phi^{\prime\prime}(u)-\frac{\phi^{\prime}(u)\cdot u}{|u|^{2}}(\cdot,\cdot).
\]
\end{definition}

\begin{definition}
Let $\phi$ be a $C^{2}$-functional on $E$. For any $u\in E$ and $\theta>0,$ we define the \textit{approximate Morse index} by
\[
\tilde{m}_{\theta}(u)=\sup\left\{\dim L \,\Big{|}\, \scalebox{0.945}{ $L\underset{subspace}\subset T_{u}S_{|u|^2}: D^{2}\phi(u)[\varphi,\varphi]<-\theta\|\varphi\|^{2}, \ \forall\varphi\in L\setminus\{0\}$} \right\}.
\]
If $u$ is a critical point for the constrained functional $\phi|_{S_{\mu}}$ and $\theta=0,$ we say that this is the \textit{Morse index of $u$ as constrained critical point}.
\end{definition}

\begin{definition}
   Let $\phi$ be a $C^{2}$-functional on $E$. For any critical point $u\in E,$ we define the \textit{Morse index} by
\[
 m(u)=\sup\left\{\dim L \,\Big|\, L\underset{subspace}\subset E\text{ such that}: D^{2}\phi(u)[\varphi,\varphi]<0, \ \forall\varphi\in L\setminus\{0\} \right\}.
\] 
\end{definition}
The following theorem was established in \cite{borthwick2024bounded}.

\begin{theorem}\label{monotonous}
Let \( I \subset (0, +\infty) \) be an interval and consider a family of \( C^2 \) functionals \(\Phi_\rho : E \to \mathbb{R} \) of the form
\[
\Phi_\rho(u) = A(u) - \rho B(u), \quad \rho \in I,
\]
where \( B(u) \geq 0 \) for every \( u \in E ,\) and
\begin{equation}\label{assumption2.2}
\text{either } A(u) \to +\infty \text{ or } B(u) \to +\infty \text{ as } u \in E \text{ and } \|u\| \to +\infty.
\end{equation}
Suppose moreover that \(\Phi'_\rho\) and \(\Phi''_\rho\) are \(\xi\)-H\"older continuous on bounded sets for some \(\xi \in (0, 1]\). Finally, suppose that there exist \( w_1, w_2 \in S_\mu \) (independent of \(\rho\)) such that, setting
\[
\Gamma = \left\{ \gamma \in C([0, 1], S_\mu) \middle| \gamma(0) = w_1, \quad \gamma(1) = w_2 \right\},
\]
we have
\[
c_\rho := \inf_{\gamma \in \Gamma} \max_{t \in [0, 1]} \Phi_\rho(\gamma(t)) > \max\{\Phi_\rho(w_1), \Phi_\rho(w_2)\}, \quad \rho \in I.
\]
Then, for almost every \(\rho \in I,\) there exist sequences \(\{u_n\} \subset S_\mu\) and \(\{\zeta_n\}\subset \mathbb{R}^+\) with \(\zeta_n\to 0\) such that, as \( n \to +\infty ,\)

\smallskip
\((i)\ \Phi_\rho(u_n) \to c_\rho;\)

\smallskip
\((ii)\ \|\Phi'_\rho|_{S_\mu}(u_n)\|_* \to 0;\)

\smallskip
\((iii)\ \{u_n\}\  \text{ is bounded in } E;\)

\smallskip

\((iv)\ \tilde{m}_{\zeta_n}(u_n) \leq 1.\)
\smallskip
\end{theorem}
It is easily observed, see \cite[Remarks 1.3]{borthwick2024bounded}, from Theorem \ref{monotonous} $(ii)-(iii)$ that
\[
\Phi_{\rho}^{\prime}(u_{n}) + \lambda_{n}(u_{n}, \cdot) \to 0 \text{ in } E^{\prime} \text{ as } n \to +\infty
\]
where we have set
\[
\lambda_{n} := -\frac{1}{\mu}(\Phi_{\rho}^{\prime}(u_{n}) \cdot u_{n}).
\] 
Also, Theorem \ref{monotonous} $(iv)$ directly implies that if there exists a subspace \(W_{n} \subset T_{u_{n}}S_{\mu}\) such that
\[
D^{2}\Phi_{\rho}(u_{n})[w, w] = \Phi_{\rho}^{\prime\prime}(u_{n})[w, w] + \lambda_{n}(w, w) < -\zeta_{n}||w||^{2}, \quad \text{for all } w \in W_{n} \setminus \{0\},
\] 
then necessarily \(\dim W_{n} \leq 1\). We called the sequence \(\{\lambda_{n}\} \subset \mathbb{R}\) the sequence of \textit{almost Lagrange multipliers}.

We shall see if \(\{u_{n}\} \subset S_{\mu}\) converges to some \(u \in S_{\mu}\) then information on the Morse index of \(u\) as constrained critical point can be obtained. Firstly, similar to \cite[Lemma 2.5]{Borthwick_2023}, we obtain results in the high-dimensional case, which are used to derive information about the sequence of \textit{almost Lagrange multipliers} $\{\lambda_n\}$.

\begin{lemma}\label{lemma2.5} Let \(\{u_n\} \subset S_\mu,\) \(\{\lambda_n\} \subset \mathbb{R}\) and \(\{\zeta_n\} \subset \mathbb{R}^+\) with \(\zeta_n \to 0\). Assume that the following conditions hold$:$
\begin{enumerate}
\item[(i)] For large enough \(n \in \mathbb{N},\) all subspaces \(W_n \subset E\) with the property  
\begin{equation}\label{2.7}
\Phi_\rho''(u_n)[\varphi, \varphi] + \lambda_n |\varphi|^2 < -\zeta_n \|\varphi\|^2, \quad \text{for all } \varphi \in W_n \setminus \{0\},
\end{equation}
satisfy: \(\dim W_n \leq 2\).

\item[(ii)] There exist \(\lambda \in \mathbb{R},\) a subspace \(Y\) of \(E\) with \(\dim Y \geq 3\) and \(a > 0\) such that, for large enough \(n \in \mathbb{N},\)  
\begin{equation}\label{2.8}
\Phi_\rho''(u_n)[\varphi, \varphi] + \lambda|\varphi|^2 \leq -a\|\varphi\|^2, \quad \text{for all } \varphi \in Y.
\end{equation}
\end{enumerate}
Then \(\lambda_n > \lambda\) for all large enough \(n \in \mathbb{N}\). In particular, if $($\ref{2.8}$)$ holds for any \(\lambda < 0,\) then \(\liminf_{n \to \infty} \lambda_n \geq 0\).
\end{lemma}

Next, we prove that under suitable conditions, we only need weak convergence to derive the Morse index information of the weak limit. Independently, L.Appolloni and R.Molle (2025) recently demonstrated a similar result in the concrete case of \cite[Proposition 3.5]{appolloni2025normalizedschrodingerequationsmasssupercritical}. Our work in Theorem \ref{keepappromaxiateindex} reveals that the core insight applies more generally. 

For ease of presentation, we make the following assumptions.
Let $S$ be a densely embedded subspace of $E$ and $\Phi\in 
 C^2(E,\mathbb{R}),$ 
 \begin{enumerate}
     \item[$(C)$] $\forall w\in S,$ $\Phi^{\prime}(\cdot)(w)$ and
 $ \Phi^{\prime\prime}(\cdot)[w,w]$ are weakly sequentially continuous,
 \item[$(U)$] $ \Phi^{\prime}(u)(\bar{w}_m)\xrightarrow{m \to +\infty} \Phi^{\prime}(u)(w)$ and $ \Phi^{\prime\prime}(u)[\bar{w}_m,\bar{w}_m]\xrightarrow{m \to +\infty} \Phi^{\prime\prime}(u)[w,w]$ hold uniformly on any bounded set when $\{\bar{w}_m\}_m\subset S$ is a dense approximation of $w\in E,$ $i.e.,$ $w=\lim\limits_{m\to\infty}\bar{w}_m.$
 \end{enumerate}
\begin{theorem}  \label{keepappromaxiateindex}   
Let $\mu>0,$ $S$ be a densely embedded subspace of $E,$ and $\Phi\in 
 C^2(E,\mathbb{R})$ satisfying $(C)$ and $(U)$.
   Assume $\{u_{n}\}\subset S_\mu$ is a bounded sequence such that 
\begin{equation}\label{2.4}
    \Phi^{\prime}(u_{n}) + \lambda_{n}(u_{n}, \cdot) \xrightarrow{n \to +\infty} 0 \text{ in } E^{\prime} \text{ where }\lambda_n= -\frac{1}{|u_n|^2}(\Phi^{\prime}(u_n),u_n)
\end{equation}
and $u_{n}\rightharpoonup u\not=0$ on $E$. If there exists a positive sequence $\{\zeta_{n}\}$ with $\zeta_{n}\xrightarrow{n\rightarrow \infty}0,$ such that $\tilde{m}_{\zeta_{n}}(u_{n})\leq M$. Then $\tilde{m}_{0}(u)\leq M.$ 
\begin{proof}
Otherwise, there exists a $W_{0}\subset T_{u}S_{|u|^2}$ with $\dim W_{0}=M+1$ such that
\begin{align*}
  D^{2}\Phi(u)[w,w]<0,\  \forall w\in W_{0}\backslash\{0\}.  
\end{align*}
Since the dimension of $W_{0}$ is finite and $D^{2}\Phi(u)$ is bilinear, we deduce that there exists $\beta>0$ such that 
\begin{align*}
    D^{2}\Phi(u)[w,w]<-\beta ||w||^{2},\ \forall w\in W_{0}\backslash\{0\}.
\end{align*}
$D^{2}\Phi(u)[w,w]= \Phi^{\prime\prime}(u)[w, w] + \lambda(w, w),$ where $\lambda=-\frac{1}{|u|^2}(\Phi^{\prime}(u),u)$. Since $\{u_n\}$ is bounded and $\Phi\in C^2,$ $\lim\limits_{n\rightarrow\infty}\lambda_n$ exists up to a subsequence. By condition $(C),$ $(U)$ and $($\ref{2.4}$),$ extracting a subsequence if necessary we have 
\begin{align*}
\Phi^{\prime}(u) + (\lim\limits_{n\to\infty}\lambda_{n})(u, \cdot) = 0 \text{ in } E^{\prime},\  \text{and}\ (\Phi^{\prime\prime}(u_n)-\Phi^{\prime\prime}(u))[w, w]\xrightarrow{n\to\infty} 0,
\end{align*}
thus $\lambda=\lim\limits_{n\to\infty}\lambda_{n}.$ 
And for $n\in \mathbb{N}$ large enough,
\begin{align}   
&D^{2}\Phi(u_{n})[w,w]=\Phi^{\prime\prime}(u_n)[w, w] + \lambda_n(w, w)\notag\\
&=\Phi^{\prime\prime}(u)[w, w] + \lambda(w, w)+(\Phi^{\prime\prime}(u_n)-\Phi^{\prime\prime}(u))[w, w] + (\lambda_n-\lambda)(w, w)\notag\\
&\leq -\frac{\beta}{2}||w||^{2}, \forall w\in W_{0}\backslash\{0\}\label{wbeta}.
\end{align}
For any $w\in W_0\backslash\{0\},$ there is a decomposition 
\begin{align}\label{decomposition}
    w=w_{n}+l_{n}u_{n},\quad\text{where}\quad w_{n}\in T_{u_{n}}S_{|u_n|^2},l_{n}\in \mathbb{R}.
\end{align}  $u_{n}$ weakly converges to $u$ in $H,$ so we have \begin{equation}\label{lnto0}l_{n}=\frac{(w,u_n)}{|u_n|^2}\rightarrow 0,\end{equation} and hence  
\begin{equation}\label{wntow}||w-w_{n}||^2=l_n^2||u_n||^2\xrightarrow{n\rightarrow \infty} 0,\quad i.e.,\  w_n\to w \quad\text{in} \quad E.\end{equation}  
For $n$ large enough, define the projection mapping
\begin{align*}
    P_n: E&\to T_{u_{n}}S_{|u_n|^2}\\
    u&\mapsto P_nu:(P_nu,u_n)=0,
\end{align*}
and consider the restricted mapping:
\begin{align*}
    P_n|_{W_0}: W_0&\to P_n W_0:=W_n\\
    w&\mapsto P_nw=w_n.
\end{align*}
 It is easy to observe that $P_n|_{W_0}$ is a surjection. For $w^1\not=w^2$ in $W_0,$ we have decompositions 
 \begin{align*}
     w^i=P_nw^i+l_n^iu_n,\ i=1,2,
 \end{align*}
  hence \begin{align*}
     0=(w^1,u)=(P_nw^1,u)+l_n^1(u_n,u),\\
     0=(w^2,u)=(P_nw^2,u)+l_n^2(u_n,u),
 \end{align*}
if $P_nw^1=P_nw^2,$ then $l_n^1(u_n,u)=l_n^2(u_n,u),$ and because $(u_n,u)\to|u|^2\not=0,$ so for $n$ large enough, we have $l_n^1=l_n^2$ if $P_nw^1=P_nw^2.$ That is, $P_n|_{W_0}$ is injection for $n$ large enough.
 So for $n $ large enough, the mapping $P_n|_{W_0}$ defines a bijection from $W_{0}$ to the subspace $W_{n}=P_nW_0$ of $T_{u_{n}}S_{\mu},$ so $\dim W_{n}=M+1$. By $($\ref{wbeta}$),$ $($\ref{decomposition}$),$ $($\ref{lnto0}$)$ and $($\ref{wntow}$),$ we deduce that
\begin{align*}
    D^{2}\Phi(u_{n})[w_{n},w_{n}]&=D^{2}\Phi(u_{n})[w,w]-2l_{n}D^{2}\Phi(u_{n})[w_{n},u_{n}]\\
    &\quad+(l_{n})^{2}D^{2}\Phi(u_{n})[u_{n},u_{n}]\\
    &\leq -\frac{\beta}{2}||w||^{2}-2l_{n}D^{2}\Phi(u_{n})[w_{n},u_{n}]+(l_{n})^{2}D^{2}\Phi(u_{n})[u_{n},u_{n}]\\
    &= -\frac{\beta}{2}(||w_{n}||^{2}+2l_{n}\langle w_{n},u_{n}\rangle+(l_{n})^{2}||u_{n}||^{2})\\
    &  \quad    -2l_{n}D^{2}\Phi(u_{n})[w_{n},u_{n}]+(l_{n})^{2}D^{2}\Phi(u_{n})[u_{n},u_{n}]\\
    &\leq -\frac{\beta}{4}||w_{n}||^{2}, \ \forall w_{n}\in W_{n}\backslash\{0\}.
\end{align*}
It contradicts $\tilde{m}_{\zeta_{n}}(u_{n})\leq M.$ Thus, we infer $\tilde{m}_{0}(u)\leq M.$ 
\end{proof}
\end{theorem}
\begin{theorem}\label{unot=0}
    Let $S$ be a densely embedded subspace of $E$ and $\Phi$ be a $C^2-$ functional on $E$ of the form $\Phi(u)=A(u)-B(u)$ which $A(u)$ and $B(u)$ satisfy $(C)$ and $(U),$ we make the following assumption:
\begin{enumerate}
    \item [$(A)$]there is a $\alpha_1>0$ such that $0\leq A(u)\leq \alpha_1 A^\prime(u)u$ and $A^\prime(\cdot)\cdot:E\to \mathbb{R}$ is weakly lower semicontinuous, i,e, $u_n\rightharpoonup u$ in $E$ implies that $A^\prime(u)u\leq\liminf\limits_{n\to\infty}A^\prime(u_n)u_n.$
    \item [$(B)$]$B(u)\geq0,$ and $B^\prime(\cdot)\cdot \ : E\to \mathbb{R}$ is weakly sequentially continuous, i,e, $u_n\rightharpoonup u$ in $E$ implies that $B^\prime(u_n)u_n\to B^\prime(u)u$ as $n\to\infty.$
\end{enumerate}
If there are a bounded sequence $\{u_n\}$ and a non-negative  bounded sequence $\{\lambda_n\}$ such that 
   \begin{align}\label{2.4again}
   c:=\lim\limits_{n\to\infty}\Phi(u_n)>0 \quad\text{and} \quad 
  \Phi^{\prime}(u_{n}) + \lambda_{n}(u_{n}, \cdot) \xrightarrow{n \to +\infty} 0 \text{ in } E^{\prime}.
\end{align}
Then there is a $u\not=0$ such that $u_n\rightharpoonup u$ as $n\to \infty.$
\begin{proof}
    $u_n$ is bounded, so there is $u\in E$ such that $u_n\rightharpoonup u,$ we just need to prove $u\not=0.$ The condition $($\ref{2.4again}$)$ implies that 
    \[
    A^{\prime}(u_{n})u_n + \lambda_{n}(u_{n}, u_n)-B^{\prime}(u_{n})u_n\rightarrow0
    \]
    and for some $\lambda\geq0,$ it holds that
    \[A^{\prime}(u) + \lambda(u, \cdot)-B^{\prime}(u)=0\quad \text{in}\quad E^\prime.\]
    Then considering condition $(A),$ for $n$ large enough, it holds that 
    \begin{align}\label{B}
        0\leq A^{\prime}(u_{n})u_n-A^{\prime}(u)u + \lambda(u_n-u, u_n-u)&=B^{\prime}(u_{n})u_n-B^{\prime}(u)u+o(1),
    \end{align}
    the right side of $($\ref{B}$)$ tends to 0 by the assumption condition $(B),$ so we deduce that $A^{\prime}(u_{n})u_n\to A^{\prime}(u)u.$
And then \begin{equation*}
        0<c=\lim\limits_{n\to \infty} A(u_n)-B(u_n)\leq \alpha_1 A^\prime(u_n)u_n\to \alpha_1A^\prime(u)u,
    \end{equation*}
    so $u\not=0,$ we complete this proof.
\end{proof}
\end{theorem}

\section{Mountain pass solution for approximating problems }\label{sec3}
In the context of this paper, we have $E=H_0^{1}\left(\Omega\right)$ and $H=L^{2}\left(\Omega\right)$.
In the first part of the section, we collect some facts concerning the NLS equation, which is inspired by \cite{bartsch2016normalized}. Let us consider the problem
\begin{equation}\label{tag2.1}
\begin{cases}
-\Delta w + w = w^{p-1} & \text{in } \mathbb{R}^N,\\
w > 0 & \text{in } \mathbb{R}^N, \\
w(0) = \max_{\mathbb{R}^N} w \quad \text{and} \quad w \in H^1(\mathbb{R}^N).
\end{cases}
\end{equation}
For $p\in (2,2^*)$ fixed.
It is well known that $($\ref{tag2.1}$)$ has a unique solution which is radial,  denoted by \( w_0 \). In what follows, we set
\begin{equation}\label{tag2.2}
C_0 := \int_{\mathbb{R}^N} w_0^2 \quad \text{and} \quad C_1 := \int_{\mathbb{R}^N} w_0^p.
\end{equation}

\begin{lemma}
    $\int_{\mathbb{R}^N}|\nabla w_{0}|^2=\gamma_pC_1,$ where $\gamma_p=\frac{N(p-2)}{2p}.$
    \begin{proof}
        Combing $($\ref{tag2.1}$)$ and the Pohozaev identity, the identities hold
        $$\begin{cases}
           \int_{\mathbb{R}^N} |\nabla w_0|^2+\int_{\mathbb{R}^N} w_0^2=\int_{\mathbb{R}^N} w_0^p,\\
           \frac{N-2}{2}\int_{\mathbb{R}^N} |\nabla w_0|^2+\frac{N}{2}\int_{\mathbb{R}^N} w_0^2=\frac{N}{p}\int_{\mathbb{R}^N} w_0^p.
        \end{cases}$$
        So $ \int_{\mathbb{R}^N} |\nabla w_0|^2=\gamma_p\int_{\mathbb{R}^N} w_0^p.$ 
    \end{proof}
\end{lemma}

For \( b, \rho \in \mathbb{R} \) fixed, let us search for \( (\lambda, w) \in \mathbb{R} \times H^1(\mathbb{R}^N) ,\) with \( \lambda > 0 \) in \( \mathbb{R}^N ,\) solving
\begin{equation}\label{tag2.3}
\begin{cases}
-\Delta w + \lambda w = \rho w^{p-1} & \text{in } \mathbb{R}^N, \\
w(0) = \max w \quad \text{and} \quad \int_{\mathbb{R}^N} w^2 = b^2.
\end{cases}
\end{equation}
Solutions \( w \) of $($\ref{tag2.3}$)$ can be found as critical points of \( I_\rho: H^1(\mathbb{R}^N) \mapsto \mathbb{R} ,\) defined by
\begin{equation}\label{tag2.4}
I_\rho(w) = \int_{\mathbb{R}^N} \left( \frac{1}{2} |\nabla w|^2 - \frac{\rho}{p} w^p \right),
\end{equation}
constrained on the \( L^2 \)-sphere \( T_b:=\{w\in H^1(\mathbb{R}^N):\int_{\mathbb{R}^N}w^2=b^2\} \) and \( \lambda \) appears as Lagrange multipliers. It is well known that they can be obtained by the solutions of $($\ref{tag2.1}$)$ by scaling.

Let us introduce the set
\begin{equation}\label{tag2.5}
\mathcal{P}(b, \rho) := \left\{ w \in T_b : \int_{\mathbb{R}^N} |\nabla w|^2 = \rho\gamma_p \int_{\mathbb{R}^N} w^p \right\}.
\end{equation}
The role of \( \mathcal{P}(b, \rho) \) is clarified by the following result.

\begin{lemma}\label{lemma2.1}
If \( w \) is a solution of $($\ref{tag2.3}$),$ then \( w \in \mathcal{P}(b, \rho) \). In addition, the positive solution \( w \) of $($\ref{tag2.3}$)$ minimizes \( I_\rho \) on \( \mathcal{P}(b, \rho) \).
\end{lemma}

\begin{proof}
The proof of the first part is a simple consequence of the Pohozaev identity. We refer to Lemma 2.7 in \cite{jeanjean1997existence} for more details. For the last part, we refer to Lemma 2.10 in \cite{jeanjean1997existence}.
\end{proof}

\begin{proposition}\label{lambda_{b,}}
Problem $($\ref{tag2.3}$)$ has a unique positive solution \( (\lambda_{b,\rho}, w_{b,\rho}) \) defined by

\begin{align}
\lambda_{b,\rho} := (\frac{C_0}{\rho^{\frac{2}{p-2}} b^2})^{\frac{2(p-2)}{N(p-2)-4}} \quad\text{and}\quad
w_{b,\rho}(x) := (\frac{C_0}{\rho^{\frac{N}{2}} b^2})^{\frac{p-2}{N(p-2)-4}} w_0 \left( \frac{C_0}{\rho^{\frac{2}{p-2}} b^2} x \right).
\end{align}
The function \( w_{b,\rho} \) satisfies
\begin{equation}\label{tag2.6}
\int_{\mathbb{R}^N} |\nabla w_{b,\rho}|^2 = \gamma_p\rho^{-\frac{2}{p-2}}(\frac{C_0}{\rho^{\frac{2}{p-2}} b^2})^{\frac{p-p\gamma_p}{p\gamma_P-2}}C_1,
\end{equation}

\begin{equation}\label{tag2.7}
\int_{\mathbb{R}^N} w_{b,\rho}^p =\rho^{-\frac{p}{p-2}}(\frac{C_0}{\rho^{\frac{2}{p-2}} b^2})^{\frac{p-p\gamma_p}{p\gamma_P-2}}C_1,
\end{equation}

\begin{equation}\label{tag2.8}
l(b,\rho):=I_\rho(w_{b,\rho}) = \frac{p\gamma_p-2}{2p}\rho^{-\frac{2}{p-2}}(\frac{C_0}{\rho^{\frac{2}{p-2}} b^2})^{\frac{p-p\gamma_p}{p\gamma_P-2}}C_1.
\end{equation}
The value \( I_\rho(w_{b,\rho}) \) is called the least energy level of problem $($\ref{tag2.3}$)$. Moreover, there exists \( c_0 > 0 \) $($see \cite{bahri1997existence} and the references therein$)$ such that
\begin{equation}\label{tag2.7'}
w_{b,\rho}(x)|x|^{\frac{N-1}{2}} e^{\sqrt{\lambda_{b,\rho}}|x|} \to c_0, \quad \text{as } |x| \to \infty,
\end{equation}
\begin{equation}\label{tag2.8'}
w'_{b,\rho}(r) r^{\frac{N-1}{2}} e^{\sqrt{\lambda_{b,\rho}} r} \to -c_0 \sqrt{\lambda_{b,\rho}}, \quad \text{as } r = |x| \to +\infty.
\end{equation}
\end{proposition}

\begin{proof}
It is not difficult to directly check that \( w_{b,\rho} \) defined in the proposition is a solution of $($\ref{tag2.3}$)$ for \( \lambda = \lambda_{b,\rho} >0 \). By \cite{Kwong1989UniquenessOP}, it is the only positive solution. To obtain $($\ref{tag2.6}$),$ $($\ref{tag2.7}$)$ and $($\ref{tag2.8}$)$ we can use the explicit expression of \( w_{b,\rho} \) by a change of variables.
\end{proof}

Working with systems with several components, it will be useful to have a characterization of the best constant in a Gagliardo-Nirenberg inequality in terms of \( C_0 \) and \( C_1 \). To obtain it, we firstly observe that if \( \rho_p:=(\frac{C_0}{b^2})^{\frac{p-2}{2}}   ,\) then \( w_{b, \rho_p} \) is the unique positive solution of
\[
\begin{cases}
-\Delta w + w = (\frac{C_0}{b^2})^{\frac{p-2}{2}} w^{p-1} & \text{in } \mathbb{R}^N, \\
w(0) = \max w \quad \text{and} \quad \int_{\mathbb{R}^N} w^2 = b^2,
\end{cases}
\]
and hence is a minimizer of \( I_{\rho_p} \) on \( \mathcal{P}(b, \rho_p) \). Our next result shows that this level can also be characterized as an infimum of a Rayleigh-type quotient, defined by
\[
\mathcal{R}_b(w) := \frac{(p\gamma_p-2) \left( \int_{\mathbb{R}^N} |\nabla w|^2 \right)^{r+1}}{2p\gamma_p^{r+1} \left( (\frac{C_0}{b^2})^{\frac{p-2}{2}} \int_{\mathbb{R}^N} w^p \right)^r},
\]
where $r=\frac{2}{p\gamma_p-2}=\frac{4}{Np-2N-4}>0.$

In order to describe the minimax structure, it is convenient to introduce some notation. We define, for \( s \in \mathbb{R} \) and \( w \in H^1(\mathbb{R}^N) ,\) the radial dilation.
\[
(s \star w)(x) := e^{\frac{Ns}{2}}w(e^s x).
\]
It is straightforward to check that if \( w \in T_b ,\) then \( s \star w \in T_b \) for every \( s \in \mathbb{R} \).

\begin{lemma} 
For every $\rho>0$ fixed, and for every \(w \in T_b,\) there exists a unique \(s_{\rho,w} \in \mathbb{R}\) such that \(s_{\rho,w} \star w \in \mathcal{P}(b, \rho)\). Moreover, \(s_{\rho,w}\) is the unique critical point of \(I_{\rho},\) which is a strict maximum.
\end{lemma} 
\begin{proof}
   For every $w\in T_{b},$ 
   \[
   I_{\rho}(s\star w) =\frac{e^{2s}}{2}\int_{\mathbb{R}^N}|\nabla w|^2-\frac{\rho e^{\frac{p-2}{2}Ns}}{p}\int_{\mathbb{R}^N}|w|^p,
   \]
   \[
   \frac{\partial}{\partial s} I_{\rho}(s\star w) =e^{2s}\int_{\mathbb{R}^N}|\nabla w|^2-\rho\gamma_p e^{\frac{p-2}{2}Ns}\int_{\mathbb{R}^N}|w|^p.
   \] 
   $p>2+4/N$ follows $\frac{p-2}{2}N>2,$ so there exists a unique $s_{\rho, w}\in \mathbb{R}$ that satisfies \[
   \frac{\partial}{\partial s} I_{\rho}(s_{\rho,w}\star w)=0.
   \]
   Equivalently, \(s_{\rho,w}\) is the unique critical point of \(I_{\rho},\) which is a strict maximum.
\end{proof}

\begin{lemma}\label{lemma2.3}
There holds
\[
\inf_{w \in \mathcal{P}(b, \rho_p)} I_{ \rho_p}(w) = \inf_{T_b} \mathcal{R}_b(w).
\]

\end{lemma}

\begin{proof}
If \(w \in \mathcal{P}(b,\rho_p),\) then
\[
\frac{ \int_{\mathbb{R}^N} |\nabla w|^2 }{(\frac{C_0}{b^2})^{\frac{p-2}{2}} \gamma_p \int_{\mathbb{R}^N} w^p} = 1 \quad \text{and} \quad I_{\rho_p}(w) = (\frac{p\gamma_p-2}{2p\gamma_p}) \int_{\mathbb{R}^N} |\nabla w|^2.
\]
Therefore
\[
\quad I_{\rho_p}(w) = (\frac{p\gamma_p-2}{2p\gamma_p}) \int_{\mathbb{R}^N} |\nabla w|^2\left( \frac{ \int_{\mathbb{R}^N} |\nabla w|^2 }{(\frac{C_0}{b^2})^{\frac{p-2}{2}} \gamma_p \int_{\mathbb{R}^N} w^p} \right)^{\frac{2}{p\gamma_p-2}} = \mathcal{R}_b(w),
\]
which proves that \(\inf_{ \mathcal{P}(b,\rho_p)} I_{ C_0/b^2}(w) \geq \inf_{T_b} \mathcal{R}_b(w)\). On the other hand, it is easy to check that
\[
\mathcal{R}_b(s \star w) = \mathcal{R}_b(w) \quad \text{for all } s \in \mathbb{R}, \quad w \in T_{b}.
\]
By the above lemma, we conclude that
\[
\mathcal{R}_b(w) = \mathcal{R}_b(s_{\rho_p,w} \star w) = I_{ \rho_p}(s_{\rho_p,w} \star w) \geq \inf_{ \mathcal{P}(b, \rho_p)} I_{ \rho_p}(w),
\]
for every \(w \in T_{b}\).

\end{proof}

Let us recall the following Gagliardo-Nirenberg inequality: there exists a universal constant \( C_{p,N} > 0 \) such that
\begin{equation}\label{tag2.9}
\int_{\mathbb{R}^N} w^p \leq C_{p,N} \left( \int_{\mathbb{R}^N} w^2 \right)^{\frac{p(1-\gamma_p)}{2}} \left( \int_{\mathbb{R}^N} |\nabla w|^2 \right)^{\frac{p\gamma_p}{2}} \quad \text{for all } w \in H^1(\mathbb{R}^N).
\end{equation}

In particular, the optimal value of \( C_{p,N} \) can be found as
\begin{align}
\frac{1}{C_{p,N}^r} &= \inf_{w \in H^1(\mathbb{R}^N) \setminus \{0\}} \frac{\left( \int_{\mathbb{R}^N} w^2 \right)^{\frac{p(1-\gamma_p)}{2}r} \cdot \left( \int_{\mathbb{R}^N} |\nabla w|^2 \right)^{r+1}}{\left( \int_{\mathbb{R}^N} w^p \right)^r} \\
&= \inf_{w \in T_b} \frac{b^{p(1-\gamma_p)r} \left( \int_{\mathbb{R}^N} |\nabla w|^2 \right)^{r+1}}{\left( \int_{\mathbb{R}^N} w^p \right)^r},\label{tag2.10}
\end{align}
where the last equality comes from the fact that the ratio on the last algebraic fraction is invariant with respect to multiplication of \( w \) with a positive number.

\begin{lemma}\label{C_{p,N}}
In the previous notation, we have
\begin{equation}
C_{p,N}^r = \gamma_p^{\frac{-p\gamma_p}{p\gamma_p-2}}C_0^{\frac{p\gamma_p-p}{p\gamma_P-2}}C_1^{-1},\ \ C_{p,N}=\gamma_p^{\frac{-p\gamma_p}{2}}C_0^{\frac{p\gamma_p-p}{2}}C_1^{-\frac{p\gamma_{p}-2}{2}}
\end{equation}
where \( C_0 \) and \( C_1 \) have been defined in $($\ref{tag2.2}$)$.
\end{lemma}

\begin{proof}
Multiplying and dividing the last term in $($\ref{tag2.10}$)$ by \( \frac{p\gamma_p-2}{2p\gamma_p^{r+1}} b^{2r}C_0^{-r},\) we deduce that
\begin{equation*}
\frac{1}{C_{p,N}^r} = \frac{2p\gamma_p^{r+1}}{p\gamma_p-2}b^{p(1-\gamma_p)r}(\frac{C_0}{b^2})^{\frac{(p-2)r}{2}} \inf_{w \in T_b} \mathcal{R}_b(w).
\end{equation*}
Hence, by Proposition \ref{lambda_{b,}} and Lemma \ref{lemma2.3}, we infer that
\[
\frac{1}{C_{p,N}^r} = \frac{2p\gamma_p^{r+1}}{p\gamma_p-2}b^{p(1-\gamma_p)r}(\frac{C_0}{b^2})^{\frac{(p-2)r}{2}}  I_{\rho_p}(w_{b, \rho_p}) = \gamma_p^{\frac{p\gamma_p}{p\gamma_P-2}}C_0^{\frac{p-p\gamma_p}{p\gamma_P-2}}C_1.
\]
\end{proof}
Next, motivated by \cite {zhang2022normalized}, we introduce the map \( \Psi_{a,b,R} : \mathbb{R}^N \to T_a \) defined by
\[
\Psi_{a,b,R}(y) = \frac{a}{\| \varphi(\cdot / R) w_{b,\rho}(\cdot - y) \|_{L^2(\mathbb{R}^N)}} \varphi(\cdot / R) w_{b,\rho}(\cdot - y), \quad \forall y \in \mathbb{R}^N,
\]
where \( a,b, R > 0, 2 < p < 2^*,\) and \( \varphi : \mathbb{R}^N \to [0,1] \) is a smooth radial function satisfying
\[
\varphi(x) \equiv 0 \text{ if } |x| \leq 1, \quad \varphi(x) \equiv 1 \text{ if } |x| \geq 2.
\]
We describe some properties of \( \Psi_{a,b,R} \) in the following lemma.

\begin{lemma}\label{Psi}
\begin{enumerate}
   \item[]
    \item[(1)] For any \( R > 0 ,\) \( \Psi_{a,b,R} \in C(\mathbb{R}^{N}, T_{a}) \)$;$
    \item[(2)] \( \lim\limits_{|y|\to +\infty} I_\rho(\Psi_{a,b,R}(y)) = \frac{1}{2}\frac{a^2}{b^2}\int_{\mathbb{R}^N}|\nabla w_{b,\rho}|^2-\frac{1}{p}\frac{a^p}{b^p}\int_{\mathbb{R}^N}| w_{b,\rho}|^p \) uniformly for \( R,b > 0 \) bounded.
\end{enumerate}
\end{lemma}

\begin{proof}
(1) This assertion is obviously true.

    (2)Let \( 0 < R \leq R_{0} < +\infty \) and \( 0 < b \leq b_{0} < +\infty \). For any \( s \in [2, 2^{*}) ,\) we firstly claim that
    \begin{equation}\label{tag3.6}
    \| \varphi(\cdot / R) w_{b,\rho}(\cdot - y) \|_{L^{s}(\mathbb{R}^{N})} \to \| w_{b,\rho} \|_{L^{s}(\mathbb{R}^{N})}, \quad \text{as } |y| \to +\infty
    \end{equation}
    uniformly for \( 0 < R \leq R_{0} \) and \( 0 < b \leq b_{0} < +\infty \).
    
Indeed, by the definition of \( \varphi \) and $($\ref{tag2.7'}$),$ for \( |y| \) large, we have
    \begin{align}
    \int_{\mathbb{R}^{N}} |\varphi(x / R)& w_{b,\rho}(x - y) - w_{b,\rho}(x - y)|^{s} \, dx \notag\\
    &\leq C \int_{B(0,2R)} |w_{b,\rho}(x - y)|^{s} \, dx \notag\\
    &\leq C' \int_{B(0,2R)} |x - y|^{-\frac{N-1}{2} s} e^{-\sqrt{\lambda_{b,\rho}} |x - y| s} \, dx \notag\\
    &\leq C'' R_{0}^{N} (\sqrt{\lambda_{b,\rho}} |y|)^{-s}\notag\\
    &=C'' R_{0}^{N} (\frac{\rho^{\frac{2}{p-2}}b^2}{C_0})^{\frac{s(p-2)}{N(p-2)-4}} |y|^{-s},\notag
    \end{align}
    where \( C, C', C'' \) are positive constants independent of \( y \). Note that
    \begin{align*}
    \| w_{b,\rho} \|_{L^{s}(\mathbb{R}^{N})} &- \| \varphi(\cdot / R) w_{b,\rho}(\cdot - y) - w_{b,\rho}(\cdot - y) \|_{L^{s}(\mathbb{R}^{N})} \\
    &\leq \| \varphi(\cdot / R) w_{b,\rho}(\cdot - y) \|_{L^{s}(\mathbb{R}^{N})} \\
    &\leq \| w_{b,\rho} \|_{L^{s}(\mathbb{R}^{N})} + \| \varphi(\cdot / R) w_{b,\rho}(\cdot - y) - w_{b,\rho}(\cdot - y) \|_{L^{s}(\mathbb{R}^{N})}.
    \end{align*}
    The claim is obtained immediately.

    Similarly, by $($\ref{tag2.7'}$),$ $($\ref{tag2.8'}$)$ and the definition of \( \varphi ,\) for \( |y| \) large, we get
    \begin{align*}
    \int_{\mathbb{R}^{N}}& |\nabla[\varphi(x / R) w_{b,\rho}(x - y)] - \nabla w_{b,\rho}(x - y)|^{2} \, dx \\
    &\leq C \frac{1}{R^{2}} \int_{B(0,2R)} |w_{b,\rho}(x - y)|^{2} \, dx + C \int_{B(0,2R)} |\nabla w_{b,\rho}(x - y)|^{2} \, dx \\
    &\leq C' R_{0}^{N-2} (\sqrt{\lambda_{b,\rho}} |y|)^{-2} + C' R_{0}^{N} (\sqrt{\lambda_{b,\rho}} |y|)^{-2},
    \end{align*}
    where \( C, C' \) are positive constants independent of \( y \). This implies that
    \begin{equation} \label{tag3.7}
    \| \nabla[\varphi(\cdot / R) w_{b,\rho}(\cdot - y)] \|_{L^{2}(\mathbb{R}^{N})} \to \| \nabla w_{b,\rho} \|_{L^{2}(\mathbb{R}^{N})}
   \end{equation}
    as \( |y| \to +\infty \) uniformly for \( 0 < R \leq R_{0} \) and \( 0 < b \leq b_{0} < +\infty \).
 It follows from $($\ref{tag3.6}$),$ $($\ref{tag3.7}$)$ and the definition of \( \Psi_{a,b,R} \) that
     
     \[ \lim\limits_{|y|\to +\infty} I\rho(\Psi_{a,b,R}(y)) = \frac{1}{2}\frac{a^2}{b^2}\int_{\mathbb{R}^N}|\nabla w_{b,\rho}|^2-\frac{1}{p}\frac{a^p}{b^p}\int_{\mathbb{R}^N}| w_{b,\rho}|^p 
     \]
     uniformly for \( R,b > 0 \) bounded.
\end{proof}
In the remainder of this paper, we shall denote by
$$||u||=(\frac{1}{2}\int_{\Omega}|\nabla u|^{2}+|u|^{2}dx)^{\frac{1}{2}},$$
$$||u||_{*}=(\frac{1}{2}\int_{\mathbb{R}^{N}}|\nabla u|^{2}+|u|^{2}dx)^{\frac{1}{2}},$$
$$||u||_{r}=(\int_{\Omega}|u|^{p}dx)^{\frac{1}{r}},$$
$$||u||_{*r}=(\int_{\mathbb{R}^{N}}||u|^{r}dx)^{\frac{1}{r}}$$
the norms in $H_{0}^{1}(\Omega),$ $H_{0}^{1}(\mathbb{R}^{N}),$ $L^{r}(\Omega),$ $L^{r}(\mathbb{R}^{N}),$ $r\in[2,2^*],$ respectively. The inner products in $H_{0}^{1}$ and $L^{2}$ are denoted by $\langle\cdot ,\cdot\rangle$ and $(\cdot ,\cdot)$ respectively.
Assume $\Omega\subset \mathbb{R}^{N}$ $(N\geq 3)$ is an exterior domain with $0\notin \overline \Omega,$ $\lambda\geq 0,$ $\rho\in[\frac{1}{2},1],$ $p\in(2+4/N,2^{*}),$ where $2^{*}=2N/(N-2)$ is critical Sobolev exponent.
$$E_{\rho}(u)=\frac{1}{2}\int_{\Omega}|\nabla u|^{2}dx-\frac{\rho}{p}\int_{\Omega}|x|^{-\alpha}|u|^{p}dx,$$
$$E_{\rho}^{*}(u)=\frac{1}{2}\int_{\mathbb{R}^{N}}|\nabla u|^{2}dx-\frac{\rho}{p}\int_{\mathbb{R}^{N}}|x|^{-\alpha}|u|^{p}dx,$$
$$E_{\lambda,\rho}(u)=\frac{1}{2}\int_{\Omega}|\nabla u|^{2}dx+\frac{\lambda}{2}\int_{\Omega}|u|^{2}dx-\frac{\rho}{p}\int_{\Omega}|x|^{-\alpha}|u|^{p}dx,$$
$$E_{\lambda,\rho}^{*}(u)=\frac{1}{2}\int_{\mathbb{R}^{N}}|\nabla u|^{2}dx+\frac{\lambda}{2}\int_{\mathbb{R}^{N}}|u|^{2}dx-\frac{\rho}{p}\int_{\mathbb{R}^{N}}|x|^{-\alpha}|u|^{p}dx.$$
$$S_{\mu}=\{u\in H_{0}^{1}(\Omega):\int_{\Omega}|u|^{2}dx=\mu\},$$
$$S_{\mu}^{*}=\{u\in H_{0}^{1}(\mathbb{R}^{N}):\int_{\mathbb{R}^{N}}|u|^{2}dx=\mu\}.$$
To show that the family $E_\rho$ enters into the framework of Theorem \ref{monotonous} we firstly need to show that it has a mountain pass geometry on $S_\mu$ uniformly with respect to $\rho\in[\frac{1}{2},1].$

\begin{lemma}\label{c_rho}
    For every $\mu > 0,$ there exist $w_1(\mu), w_2(\mu) \in S_\mu$ independent of $\rho \in \left[\frac{1}{2}, 1\right]$ such that
\[
c_\rho(\mu) := \inf_{\gamma \in \Gamma} \max_{t \in [0,1]} E_\rho(\gamma(t)) > \max\{E_\rho(w_1), E_\rho(w_2)\}, \quad \forall \rho \in \left[\frac{1}{2}, 1\right],
\]
where $\Gamma := \{\gamma \in C([0, 1], S_{\mu}) \mid \gamma \text{ is continuous}, \gamma(0) = w_1, \gamma(1) = w_2\}$.
\begin{proof}
    For any $\mu, k > 0,$ denote
\[
A_{\mu, k} := \{u \in S_{\mu} \mid \int_\Omega |\nabla u|^2 dx < k\},
\]
\[
\partial A_{\mu, k} := \{u \in S_{\mu} \mid \int_\Omega |\nabla u|^2 dx = k\}.
\]
Firstly note that, for any $\mu, k > 0,$ we have $A_{\mu, k} \neq \emptyset$ (and, similarly, $\partial A_{\mu, k} \neq \emptyset$). Indeed, take any function $v \in C_c^\infty(\mathbb{R})$ with $\|v\|_{L^2(\Omega)}^2 = \mu,$ and consider $v_t(x) = t^{N/2} v(tx),$ for $t > 0$. We have that
\[
\|v_{t}\|_{L^2(\Omega)}^2 = \mu \quad \text{and} \quad \|\nabla v_t\|_{L^2(\Omega)}^2 = t^2 \|\nabla v\|_{L^2(\Omega)}^2 \quad \text{for every } t > 0.
\]
Since $\Omega$ is an exterior domain, there exist $R_{0}>0,$ such that $\{x:|x|\geq R_{0}\}\subset \Omega$. Therefore, just let the support of $v$ be contained in $\{x:|x|\geq R_{0}\},$ any function $v_t$ can be regarded as a function in $S_\mu,$ and, in particular, $v_t \in A_{\mu, k}$ for $t$ small enough.
Now, we assume that $|x|\geq r_0>0$ for any $x\in \Omega$ from the assumption $0\notin \Omega,$ and by the Gagliardo–Nirenberg inequality \cite{weinstein1982nonlinear}, we obtain
\[
E_\rho(u) \geq \frac{1}{2} \|\nabla u\|_{L^2(\Omega)}^2 - \frac{r_0^{-\alpha}C_{p,N}}{p} \mu^{\frac{(1-\gamma_{p})p}{2}}\|\nabla u\|_{L^2(\Omega)}^{p\gamma_{p}}, \quad \forall u \in S_\mu. 
\]
Where $\gamma_{p}=\frac{N(p-2)}{2p}.$\\
Then, for any $\mu > 0$ and $u \in \partial A_{\mu, k_0}$ with $k_0 =  \left(\frac{r_0^\alpha p}{4C_{p,N}}\right)^{\frac{2}{p\gamma_{p}-2}} \mu^{-\frac{(1-\gamma_{p})p}{p\gamma_{p}-2}},$ by Lemma \ref{C_{p,N}}, $k_{0}$ denoted as 
\[
k_0=(r_0^\alpha)^{\frac{2}{p\gamma_p-2}}\gamma_p(\frac{p\gamma_p}{4})^{\frac{2}{p\gamma_p-2}}(\frac{C_{0}}{\mu})^{\frac{p-p\gamma_P}{p\gamma_p-2}}C_1,
\]  we have
\begin{equation}\label{tag3.2}
\inf_{u \in \partial A_{\mu, k_0}} E_\rho(u) \geq k_0 \left(\frac{1}{2} - \frac{r_0^{-\alpha}C_{p,N}}{p} \mu^{\frac{(1-\gamma_{p})p}{2}} k_0^{\frac{p\gamma_{p}-2}{2}} \right) =\frac{k_{0}}{4}=: \beta > 0, 
\end{equation}
for every $\rho \in \left[\frac{1}{2}, 1\right]$.


Moreover, let $w_{b}[y](x):=\Psi_{\sqrt{\mu},b,R_{0}}(y)\in S_{\mu},$ for $b\in\mathbb{R}^+$ and $y\in \mathbb{R}^N,$ where $R_0$ satisfies that $\mathbb{R}^N\backslash{\Omega}\subset\{x\in\mathbb{R}^N:|x|<R_{0}\}:=B(0,R_{0}).$
Connecting Proposition \ref{lambda_{b,}} and Lemma \ref{Psi}, we have 
\begin{align}
    \int_{\Omega}|\nabla w_{b}[y]|^2&=\int_{\mathbb{R}^N}|\nabla \Psi_{\sqrt{\mu},b,R_{0}}(y)|^2\rightarrow \frac{\mu}{b^2}\gamma_p\rho^{-\frac{2}{p-2}}(\frac{C_0}{\rho^{\frac{2}{p-2}}b^2})^{\frac{p-p\gamma_P}{p\gamma_p-2}}C_1, \
\end{align}
 as $|y|\rightarrow +\infty.$ We observe that for $b$ and $|y|$ large enough, we can choose a $b_1$ and a $y_1$ such that $ \int_{\Omega}|\nabla w_{b_1}[y_1]|^2<\frac{k_0}{2},$ so, $E_{\rho}(w_{b_1}[y_1])<\frac{1}{2}\int_{\Omega}|\nabla w_{b_1}[y_1]|^2<\frac{k_0}{4}=\beta.$
Next, We want to find a $b_2$ and a $y_2$ such that
\begin{equation}\label{claimofconstruct}
\int_{\Omega}|\nabla w_{b_2}[y]|^2>2k_0 \ \quad\text{and}\ \quad E_{\rho}(w_{b_2}[y])<\frac{1}{2}\beta.
\end{equation}
To estimate the value of $E_{\rho}(w_{b}[y]),$ we have
\begin{align*}
    &E_{\rho}(w_{b}[y])=\frac{1}{2}\int_{\Omega}|\nabla w_{b}[y]|^{2}dx-\frac{\rho}{p}\int_{\Omega}|x|^{-\alpha}|w_{b}[y]|^{p}dx\\
    &\leq \frac{1}{2}\int_{\Omega}|\nabla w_{b}[y]|^{2}dx-\frac{\rho(2|y|)^{-\alpha}}{p}\int_{\Omega\cap\{|x|\leq 2|y|\}}|w_{b}[y]|^{p}dx\\
    &\quad-\frac{\rho}{p}\int_{\Omega\cap\{|x|> 2|y|\}}|x|^{-\alpha}|w_{b}[y]|^{p}dx\\
    &\leq \frac{1}{2}\int_{\Omega}|\nabla w_{b}[y]|^{2}dx-\frac{\rho(2|y|)^{-\alpha}}{p}\int_{\Omega}|w_{b}[y]|^{p}dx\\
    &\quad+\frac{2\rho(2|y|)^{-\alpha}}{p}\int_{\Omega\cap\{x:|x|> 2|y|\}}|w_{b}[y]|^{p}dx.
\end{align*}
We can choose a $|y'_2| $ from the definition of $E_{\rho}(w_{b}[y])$ such that for any $y:|y|\geq |y'_2|,$ it holds that 
\[
\int_{\Omega\cap\{x:|x|> 2|y|\}}|w_{b}[y]|^{p}dx\leq\frac{1}{4}\int_{\Omega}|w_{b}[y]|^{p}dx,
\]
then for $y:|y|\geq |y'_2|,$ \[
E_{\rho}(w_{b}[y])\leq \frac{1}{2}\int_{\Omega}|\nabla w_{b}[y]|^{2}dx-\frac{\rho(2|y|)^{-\alpha}}{2p}\int_{\Omega}|w_{b}[y]|^{p}dx.
\]
Moreover, from Lemma \ref{Psi}, for $|y|$ large enough, the inequalities $($\ref{claimofconstruct}$)$ are equal to
\begin{align}\label{3.21}
   &\frac{\mu}{b^2}\gamma_p\rho^{-\frac{2}{p-2}}(\frac{C_0}{\rho^{\frac{2}{p-2}}b^2})^{\frac{p-p\gamma_P}{p\gamma_p-2}}C_1>2(r_0^\alpha)^{\frac{2}{p\gamma_p-2}}\gamma_p(\frac{p\gamma_p}{4})^{\frac{2}{p\gamma_p-2}}(\frac{C_{0}}{\mu})^{\frac{p-p\gamma_P}{p\gamma_p-2}}C_1,
   \end{align}
   and \begin{align}\label{3.22}&\quad (\frac{1}{2}\frac{\mu}{b^2}\gamma_p-\frac{(2|y|)^{-\alpha}}{2p}\frac{\mu^{\frac{p}{2}}}{b^p})\rho^{-\frac{2}{p-2}}(\frac{C_0}{\rho^{\frac{2}{p-2}}b^2})^{\frac{p-p\gamma_P}{p\gamma_p-2}}C_1\notag\\
   &\quad\quad<\frac{1}{8}(r_0^\alpha)^{\frac{2}{p\gamma_p-2}}\gamma_p(\frac{p\gamma_p}{4})^{\frac{2}{p\gamma_p-2}}(\frac{C_{0}}{\mu})^{\frac{p-p\gamma_P}{p\gamma_p-2}}C_1.
\end{align}
Furthermore, $($\ref{3.21}$)$ and $($\ref{3.22}$)$ can be reduced to 
\begin{align}
   (\frac{\mu}{\rho^{\frac{2}{p-2}}b^2})^{\frac{p-2}{p\gamma_p-2}}&>2(r_0^\alpha)^{\frac{2}{p\gamma_p-2}}(\frac{p\gamma_p}{4})^{\frac{2}{p\gamma_p-2}}, \\
    \text{and}\quad \quad (\frac{\gamma_p}{2}\frac{\mu}{b^2}-\frac{(2|y|)^{-\alpha}}{2p}\frac{\mu^{\frac{p}{2}}}{b^p})&\rho^{-\frac{2}{p-2}}(\frac{\mu}{\rho^{\frac{2}{p-2}}b^2})^{\frac{p-p\gamma_P}{p\gamma_p-2}}<\frac{1}{8}(r_0^\alpha)^{\frac{2}{p\gamma_p-2}}\gamma_p(\frac{p\gamma_p}{4})^{\frac{2}{p\gamma_p-2}}.\label{E_rho(w_by)}
\end{align}
Let $(\frac{\mu}{\rho^{\frac{2}{p-2}}b^2})^{\frac{p-2}{p\gamma_p-2}}=2s(r_0^\alpha)^{\frac{2}{p\gamma_p-2}}(\frac{p\gamma_p}{4})^{\frac{2}{p\gamma_p-2}},  \ s>1.$ So $($\ref{E_rho(w_by)}$)$ can be reduced to 
\begin{equation}8s-\rho (\frac{r_0}{2|y|})^\alpha(2s)^{\frac{p\gamma_p}{2}}<1.\end{equation}
We can choose a $b_2=b(|y|)$ such that $w_{b_2}[y](x)$ is what we are looking for.
Therefore, we can take $y_{2}$ and $b_{2}$ to make inequalities $($\ref{claimofconstruct}$)$ true, and we set $w_{1}(x):=w_{b_1}[y_1](x),$ $w_{2}(x):=w_{b_2}[y_2](x)\in S_{\mu}.$ which satisfy that
\begin{equation}\label{tag3.4}
\|\nabla w_1\|_{L^2(\Omega)}^2 <\frac{1}{2} k_0 \quad \text{and} \quad E_{\rho}(w_1) < \beta, \quad \forall \rho \in \left[\frac{1}{2}, 1\right].
\end{equation}
\begin{equation}\label{tag3.5}
\|\nabla w_2\|_{L^2(\Omega)}^2 > 2k_0 \quad \text{and} \quad E_{\rho}(w_2) <\frac{1}{2} \beta, \quad \forall \rho \in \left[\frac{1}{2}, 1\right].
\end{equation}

 Let \(\Gamma\) and \(c_{\rho}\) be defined as in the statement of the lemma for our choice of \(w_1\) and \(w_2\); the fact that \(\Gamma \neq 0\) is straightforward, since
\[
\gamma_0(t) := \frac{\mu^{1/2}}{\| (1-t)w_1 + t w_2 \|_{L^2(\Omega)}} \big( (1-t) w_1 + t w_2 \big), \quad t \in [0, 1],
\]
belongs to \(\Gamma\). By $($\ref{tag3.4}$)$ and $($\ref{tag3.5}$),$ we note that for every \(\gamma \in \Gamma\) there exists \(t_\gamma \in [0, 1]\) such that \(\gamma(t_\gamma) \in \partial A_{\mu, k_0},\) by continuity. Therefore, for any \(\rho \in \left[\frac{1}{2}, 1\right],\) we have for every \(\gamma \in \Gamma\)
\[
\max_{t \in [0,1]} E_p(\gamma(t)) \geq E_p(\gamma(t_\gamma)) \geq \inf_{u \in \partial A_{\mu, k_0}} E_p(u) \geq \beta
\]
(see \ref{tag3.2}) and thus \(C_{\rho} \geq \beta,\) while
\[
\max\{E_p(w_1), E_p(w_2)\}  < \beta.
\]
\end{proof}
\end{lemma}

\section{Compactness}\label{sec4}

\begin{lemma}\label{lemma3.2} For any \(\lambda<0,\) there exists a subspace \(Y\) of \(H_0^{1}(\Omega)\) with \(\dim Y=3\) such that
\[
\int_{\Omega}|\nabla w|^{2}\,dx+\lambda\int_{\Omega}|w|^{2}\,dx\leq\frac{\lambda}{2}\|w\|^2,\qquad\forall\,w\in Y.
\]

\begin{proof} Recalling that $B(0,R_0)^c\subset\Omega.$ Take \(\phi\in C_{0}^{\infty}(\mathbb{R}^{N})\) with \(\operatorname{supp}\phi\subset\{x\in\mathbb{R}^N:2R_0\leq|x|\leq3R_0\}\) such that \(\int_{\mathbb{R}^N}|\phi|^{2}dx=1\). We define the following functions for \(0<\sigma\leq1\):
\begin{equation*}
\psi_{j,\sigma}(x)=\sigma^{\frac{N}{2}}\phi(\sigma x-7R_0e_j),\quad \text{j=1,2,3}, 
\end{equation*}
where, $\{e_j\}_{j=1}^N$  is the standard orthonormal basis of $\mathbb{R}^N$. Clearly, for any \(0<\sigma\leq1,\) \(\int_{\Omega}|\psi_{j,\sigma}|^{2}\,dx=1\) for all \(j=1,2,3,,\) and the three functions are mutually orthogonal in \(H_0^{1}(\Omega),\) having disjoint supports.
Assume that \(w=\sum\limits_{j=1}^{3}\theta_{j}\psi_{j,\sigma}\). Then
\begin{align*}
\int_{\Omega}|\nabla w|^{2}\,dx+\lambda\int_{\Omega}|w|^{2}\,dx &=\sigma^{2}\Big{(}\sum\limits_{j=1}^{3}\theta_{j}^{2}\int_{\Omega}|\nabla\phi|^{2}\,dx\Big{)}+\lambda\Big{(}\sum\limits_{j=1}^{3}\theta_{j}^{2}\int_{\Omega}|\phi|^{2}\,dx\Big{)} \\
&=(\sigma^{2}\zeta+\lambda)\sum\limits_{j=1}^{3}\theta_{j}^{2},
\end{align*}
where \(\zeta:=\int_{\Omega}|\nabla\phi|^{2}\,dx>0\). Similarly, \(\|w\|^{2}=(\sigma^{2}\zeta+1)\sum_{j=1}^{3}\theta_{j}^{2}\). Therefore, provided that \(w\neq 0,\)
\[
\frac{\int_{\Omega}|\nabla w|^{2}\,dx+\lambda\int_{\Omega}|w|^{2}\,dx}{\|w\|^{2}}=\frac{\sigma^{2}\zeta+\lambda}{\sigma^{2}\zeta+1} \leq \frac{\lambda}{2}
\]
for every \(\sigma>0\) sufficiently small, and for every \(\theta_{1},\theta_{2},\theta_{3}\in\mathbb{R}\).
\end{proof}
\end{lemma}
The main aim of this section is to prove the following result. This is also the proof process of Theorem \ref{submass}.

\begin{proposition}\label{prop3.5}
For any $\mu>0$ and almost every $\rho\in [\frac{1}{2},1],$ there exists a positive $(\mu,\rho)$ sub-mass mountain pass solution $u_{\rho}\in H_0^1(\Omega)$ of equation $($\ref{1.1rho}$)$.
Moreover, $m(u_{\rho})\leq 2$. 
\end{proposition}

\begin{proof}
 We apply Theorem \ref{monotonous} to the family of functionals $E_{\rho},$ with $E=H_0^{1}(\Omega),$ $H=L^{2}(\Omega),$ $S_{\mu}$ and $\Gamma$ defined in Lemma \ref{c_rho}. Setting
\[
A(u)=\frac{1}{2}\int_{\Omega}|\nabla u|^{2}\,dx \quad \text{and} \quad B(u)=\frac{\rho}{p}\int_{\Omega}|x|^{-\alpha}|u|^{p}\,dx,
\]
Assumption $($\ref{assumption2.2}$)$ holds, since we have that
\[
u\in S_\mu,\,\|u\|\to+\infty \quad \implies \quad A(u)\to+\infty.
\]

Let $E^{\prime}_{\rho}$ and $E^{\prime\prime}_{\rho}$ denote respectively the first and second derivatives of $E_{\rho}$. Clearly, $E^{\prime}_{\rho}$ and $E^{\prime\prime}_{\rho}$ are both of class $C^{1},$ and hence locally H\"older continuous, on $S_\mu$.

Thus, taking into account Lemma \ref{c_rho}, by Theorem \ref{monotonous} and the considerations just after it, for almost every $\rho\in[1/2,1],$ there exist a bounded sequence $\{u_{n,\rho}\}\subset S_\mu,$ that we shall denote simply by $\{u_{n}\}$ from now on, and a sequence $\{\zeta_{n}\}\subset\mathbb{R}^{+}$ with $\zeta_{n}\to 0^{+},$ such that
\begin{equation}\label{tag3.9}
E^{\prime}_{\rho}(u_{n})+\lambda_{n}u_{n}\to 0 \quad \text{in }  (H_0^1(\Omega))',
\end{equation}
where
\begin{equation}\label{tag3.10}
\lambda_{n}:=-\frac{1}{\mu}E^{\prime}_{\rho}(u_{n})u_{n}.
\end{equation}

Moreover, if the inequality
\begin{equation}\label{tag3.11}
\int_{\Omega}\Bigl{[}|\nabla\varphi|^{2}+(\lambda_{n}-(p-1)\rho|x|^{-\alpha}|u_{n}|^{p-2})\,\varphi^{2}\Bigr{]}\,dx=E^{\prime\prime}_{\rho}(u_{n})[\varphi,\varphi]+\lambda_{n}\|\varphi\|^{2}_{2}<-\zeta_{n}\|\varphi\|^{2}
\end{equation}
holds for any $\varphi\in W_{n}\setminus\{0\}$ in a subspace $W_{n}$ of $T_{u_{n}}S_{\mu},$ then the dimension of $W_{n}$ is at most 1. In addition, since $u\in S_\mu \implies |u|\in S_{\mu},$ $w_{1},w_{2}\geq 0,$ the map $u\mapsto|u|$ is continuous, and $E_{\rho}(u)=E_{\rho}(|u|),$ it is possible to choose $\{u_{n}\}$ with the property that $u_{n}\geq 0$ on $\Omega.$

    The sequence $\{u_{n}\}$ is bounded, it follows by \ref{tag3.10} that $\lambda_n$ is bounded. Then passing to a subsequence, there exists $\lambda_\rho$ such that $\lim\limits_{n\rightarrow\infty}\lambda_{n}=\lambda_\rho$. 
        Moreover, there exists $u_{\rho}\in H_{0}^{1}(\Omega)$ such that 
        \begin{align}
            u_{n}&\rightharpoonup u_{\rho}\ in \ H_{0}^{1}(\Omega),\label{weak}\\
            u_{n}&\rightharpoonup u_{\rho}\ in \ L^{r}(\Omega),r\in[2,2^{*}],\label{weak Lp}\\
            u_{n}&\rightarrow u_{\rho} \ in \ L_{loc}^{r}(\Omega),r\in[1,2^{*}),\label{local}\\
            u_{n}&(x)\rightarrow u_{\rho}(x)\  for\ a.e.x \in \Omega,\label{a.e.}
        \end{align}
        which implies that $u_{\rho}\geq 0.$\\
     Let us firstly prove that $E_{\lambda_\rho,\rho}'(u_{\rho})=E_{\rho}'(u_{\rho})+\lambda_\rho (u_n,\cdot)=0.$ Noting that $C_{0}^{\infty}(\Omega)$ is dense in $H_{0}^{1}(\Omega),$ it suffices to check that $E_{\lambda_\rho,\rho}'(u_{\rho})\phi=0$ for all $\phi\in C_{0}^{\infty}(\Omega).$ And we have 
     \begin{align}
        E_{\lambda_n,\rho}'(u_{n})\phi&-E_{\lambda_\rho,\rho}'(u_{\rho})\phi=\int_{\Omega}\nabla(u_{n}-u_{\rho})\nabla\phi dx+\lambda_\rho\int_{\Omega}  (u_{n}-u_{\rho})\phi \notag\\
        &(\lambda_n-\lambda_\rho)\int_{\Omega}u_n\phi-\rho\int_{\Omega}|x|^{-\alpha}(|u_{n}|^{p-1}-|u_{\rho}|^{p-1})\phi dx\rightarrow 0,
     \end{align}
     which is determined by $($\ref{weak}$)$ and $($\ref{local}$)$. Thus recalling that $E_{\lambda_n,\rho}'(u_{n})\rightarrow 0,$ we indeed have $E_{\lambda_\rho,\rho}'(u_{\rho})= 0.$ Thus 
     \begin{align}\label{identical equation}
          \int_{\Omega}|\nabla u_{\rho}|^2+\lambda_\rho |u_{\rho}|^2dx=\rho\int_{\Omega}|x|^{-\alpha}|u_{\rho}|^{p} dx
     \end{align}
     
     Then we need to show that $\lambda_\rho\geq0.$ Here the Morse type information $($\ref{tag3.11}$)$ proves decisive. Since the codimension of \( T_{u_n} S_\mu \) is 1, we infer that if the inequality $($\ref{tag3.11}$)$ holds for every \( \varphi \in V_n \setminus \{0\} \) for a subspace \( V_n \) of \( H_0^1 (\Omega) ,\) then the dimension of \( V_n \) is at most 2. On the other hand, by Lemma \ref{lemma3.2}, we have that for every \( \lambda < 0 ,\) there exist a subspace \( Y \) of \( E \) with \( \dim Y = 3 \) and \( a > 0 \) such that, for \( n \in \mathbb{N} \) large,
\[
E_\rho''(u_n)[\varphi, \varphi] + \lambda \|\varphi\|_2^2 \leq \int_{\Omega} |\nabla\varphi|^2 dx + \lambda \int_{\Omega} |\varphi|^2 \leq -a \|\varphi\|^2,
\]
for all \( \varphi \in Y \setminus \{0\} \).
Therefore, Lemma \ref{lemma2.5} implies that \( \lambda_\rho \geq 0 \).

     To prove that $u_\rho\not\equiv0$. From $($\ref{tag3.9}$)$ and the fact $\lambda_n\rightarrow \lambda_\rho,$ by $E_{\lambda_\rho,\rho}'(u_{\rho})= 0,$ we deduce that when $n\to\infty,$
     \begin{align}\label{compact}
         \int_{\Omega}|\nabla(u_{n}-u_{\rho})|^2+\lambda_\rho |(u_{n}-u_{\rho})|^2dx=\rho\int_{\Omega}|x|^{-\alpha}(|u_{n}|^{p}-|u_{\rho}|^{p}) dx+o(1).
     \end{align}
    Considering $\int_{\Omega}|x|^{-\alpha}(|u_{n}|^{p}-|u_{\rho}|^{p}) dx,$ for $R>R_0$ large enough we have 
    \[
    \int_{\Omega}|x|^{-\alpha}(|u_{n}|^{p}-|u_{\rho}|^{p}) dx=(\int_{\Omega\cap\{|x|\leq R\}}+\int_{\Omega\cap\{|x|>R\}})|x|^{-\alpha}(|u_{n}|^{p}-|u_{\rho}|^{p}) dx.
    \]
So for any $k\in\mathbb{N},$ supposing the uniform bound of $\int_\Omega |u_n|^pdx$ does not exceed $M,$ there is $R_k$ such that 
\[\int_{\Omega\cap\{|x|>R\}}|x|^{-\alpha}\bigg{|}|u_{n}|^{p}-|u_{\rho}|^{p}\bigg{|} dx\leq {2M}|R_k|^{-\alpha}\leq 2^{-k},\quad \forall n\in\mathbb{N},\]
and by Rellich imbedding theorem, there is $n_k\in\mathbb{N}$ such that \[
\int_{\Omega\cap\{x:|x|\leq R_k\}}|x|^{-\alpha}\bigg{|}|u_{n_k}|^{p}-|u_{\rho}|^{p}\bigg{|} dx<2^{-k}.
\]
   Then we infer that \[
   \int_{\Omega}|x|^{-\alpha}\bigg{|}|u_{n_k}|^{p}-|u_{\rho}|^{p}\bigg{|} dx\to 0,\quad \text{as}\quad k\to \infty,
   \]
   furthermore, by $($\ref{compact}$),$ we have 
   \[
    \int_{\Omega}|\nabla(u_{n_k}-u_{\rho})|^2+\lambda_\rho |(u_{n_k}-u_{\rho})|^2dx\to 0,\quad \text{as}\quad k\to \infty.
   \]
 If $\lambda_\rho>0,$ we have $u_{n_k}\to u_\rho,$ so $u_\rho\not\equiv0$ from the condition $u_{n_k}\in S_\mu$. If $\lambda_\rho=0,$ from Theorem \ref{monotonous} $(i)$ we know 
 \begin{align}\label{c_rholambda0}
 c_\rho&=\lim\limits_{k\to\infty}\frac{1}{2}\int_{\Omega}|\nabla(u_{n_k})|^2-\frac{\rho}{p}\int_{\Omega}|x|^{-\alpha}|u_{n_k}|^{p}\notag\\
 &=\frac{1}{2}\int_{\Omega}|\nabla(u_{\rho})|^2-\frac{\rho}{p}\int_{\Omega}|x|^{-\alpha}|u_{\rho}|^{p}>\alpha>0,\end{align}
 So $u_\rho\not\equiv0$. And from Theorem \ref{keepappromaxiateindex}, we know $\tilde{m}_0(u_\rho)\leq 1.$ \\
 Moreover, when $\lambda_\rho=0,$ we also have when $k\to\infty,$
 \[
\int_{\Omega}|\nabla(u_{n_k})|^2\to\int_{\Omega}|\nabla(u_{\rho})|^2\quad\text{and}\quad \int_{\Omega}|x|^{-\alpha}|u_{n_k}|^{p}\to \int_{\Omega}|x|^{-\alpha}|u_{\rho}|^{p}.
 \]
 But there is only an upper bounded estimate for the $L^2$ norm of $u_\rho,$ that is $$\int_{\Omega}|u_{\rho}|^2\leq\int_{\Omega}|u_{n_k}|^2=\mu.$$
 Hence, $u_\rho$ is a positive $(\mu,\rho)$ sub-mass mountain pass solution for $\lambda_\rho\geq0.$ Proof completed. 
\end{proof}

\section{Blow-up analysis}\label{sec5}
From Proposition \ref{prop3.5}, there exists a sequence \( \rho_n \to 1^- ,\) and the positive $(\mu,\rho_n)$ sub-mass mountain pass solution \( u_{\rho_n} \) of \( E_{\rho_n}  \) for $\lambda_{\rho_n}\geq0,$ lying at the mountain pass level \( c_{\rho_n} ,\) with a Morse index \( m(u_{\rho_n}) \) lesser or equal to 2. To obtain the existence of solutions for equations $($\ref{1.1})-(\ref{1.2}$),$ at a strictly positive energy level, it clearly suffices to show that \(\{u_{\rho_n}\}\) converges. In this direction the key point is to show that \(\{u_{\rho_n}\}\) is bounded in \( H_0^1 (\Omega) \). For a more comprehensive discussion on the blow-up analysis, see \cite{Borthwick_2023,EspositoPetralla2011,pierotti2017normalized}.

By Lemma \ref{c_rho} and the monotonicity of \( c_\rho ,\)
\[E_1 (w_1) \leq E_\rho (w_1) \leq c_\rho \leq c_{1/2}, \quad \forall \rho \in \left[ \frac{1}{2}, 1 \right],\]
which implies that \( c_{\rho_n} \) is bounded. In addition, thanks to the condition
\[\int_\Omega \left( |\nabla u_n|^2 + \lambda_n u_n^2 \right) dx = \rho_n\int_\Omega |x|^{-\alpha}|u_n|^p dx,\]
whence it follows that
\[\left( \frac{1}{2} - \frac{1}{p} \right) \int_\Omega |\nabla u_n|^2 dx \leq c_n + \frac{\lambda_n \mu}{p}.\]
Therefore, if \(\{\lambda_n\}\) is bounded, then \(\{u_{\rho_n}\}\) is bounded as well. We shall thus assume that \(\lambda_n \to +\infty,\) up to extraction of a subsequence and, in order to reach a contradiction, develop a blow-up analysis for \(\{u_{\rho_n}\}\).

Precisely, writing for simplicity \(u_n := u_{\rho_n}\) for \(\rho_n \to 1^-\). Using standard regularity arguments, we obtain \(u_{n}\in C^{2}(\Omega)\). And by the maximum principle, we also have that $u_n>0$ in $\Omega.$ We consider the behavior of a sequence of solutions \(\{u_n\} \in H_0^1(\Omega)\) of the following problem:
\begin{equation}\label{rhoto1}\begin{cases}
-\Delta u_n + \lambda_n u_n = \rho_n |x|^{-\alpha} u_n^{p-1} & \text{on } \Omega, \\
u_n > 0 & \text{on } \Omega, \\
u_n = 0 & \text{on } \partial\Omega,
\end{cases}\end{equation}
where \(m(u_n) \leq 2\) for all \(n \in \mathbb{N},\) and it is assumed that \(\lambda_n \to +\infty\).

Firstly, we note that the local maximum points of \(u_n\) have the following estimate.

\begin{lemma}\label{localmax}
    
 Let \(x_n \in \Omega\) be a local maximum point of \(u_n\). Then 
\[u_n(x_n) \geq r_0^{\frac{\alpha}{p-2}}\lambda_n^{\frac{1}{p-2}},\]
where $r_0=\inf\limits_{\Omega}|x|.$
\end{lemma}
\begin{proof}
 By regularity, $u_n\in C^2(\Omega),$ if $x_n$ is a local maximum point of $u_n,$ then $\Delta u_n(x_n)\leq0,$ and we have 
 \[
 \Delta u_n(x_n)=\lambda_n u_n(x_n)-\rho_n |x|^{-\alpha}u_n^{p-1}(x_n)\leq 0,
 \]
 so, it holds\[
 u_n(x_n)\geq (\frac{\lambda_n|x_n|^{\alpha}}{\rho_n})^{\frac{1}{p-2}}\geq r_0^{\frac{\alpha}{p-2}}\lambda_n^{\frac{1}{p-2}}.
 \]
 
\end{proof}
As the constant $r_0$ plays no essential role in the blow-up analysis, we normalize it to be $r_0=1$ for simplicity in the following content.

\begin{lemma}\label{complication}
     Suppose that \(\lambda_{n}\to+\infty\). Let \(P_{n}\in\Omega\) be such that, for some \(R_{n}\to\infty,\)
\[|u_{n}(P_{n})|=\max_{B_{R_{n}\tilde{\varepsilon}_{n}}(P_{n})}|u_{n}(x)|~{}~{}\text{ where }~{}\tilde{\varepsilon}_{n}=|u_{n}(P_{n})|^{-\frac{p-2}{2}}\to 0.\]
Set \(\varepsilon_{n}=\lambda_{n}^{-\frac{1}{2}}\). Then
\[\left(\frac{\tilde{\varepsilon}_{n}}{\varepsilon_{n}}\right)^{2}\to\tilde{\lambda}\in(0,1].\]
and passing to a subsequence if necessary, we have

$(i)$\(P_{n}\to P\in\Omega\);

$(ii)$\(\frac{dist(P_{n},\partial\Omega)}{\tilde{\varepsilon}_{n}}\to+\infty\) as \(n\to+\infty,\) and the scaled sequence
\begin{equation}\label{scaledsequence}
v_{n}(x):=\varepsilon_{n}^{\frac{2}{p-2}}u_{n}(\varepsilon_{n}x+P_{n}) ~{}~{}\text{for }~{}x\in\Omega_n:=\frac{\Omega-P_{n}}{\varepsilon_{n}} 
\end{equation}
converges to \(W_0\in H^{1}(\mathbb{R}^{N})\) in \(C^{2}_{loc}(\mathbb{R}^{N}),\) where \(W_0\) is the unique positive solution of
\[\left\{\begin{array}[]{ll}-\Delta v+v=\Lambda_0v^{p-1}&\quad\text{ in }\mathbb{R}^{N}, \\ v(0)=\max\limits_{x\in\mathbb{R}^N}v,&\\ v(x)\to 0~{}~{}\text{as}~{}~{}|x|\to+\infty;&\end{array}\right.\]

 $(iii)$there exists \(\phi_{n}\in C^{\infty}_{0}(\Omega),\) with supp \(\phi_{n}\subset B_{R\varepsilon_{n}}(P_{n})\) for some \(R>0,\) such that 
 \[
 \int_{\Omega}|\nabla\phi_n|^2+\lambda_n\phi_n^2-\rho_n(p-1)|x|^{-\alpha}u_n^{p-2}\phi_n^2 dx<0.
 \]

$(iv)$for all \(R>0\) and \(q\geq 1,\)
\[
\lim\limits_{n\rightarrow+\infty} \lambda_n^{\frac{N}{2}-\frac{q}{p-2}} \int_{B_{R \varepsilon_n (P_n)}} u_n^{q}=\lim_{n \to +\infty} \int_{B_R(0)} v_n^{q} = \int_{B_R(0)}( W_0)^{q}.\]

\begin{proof} \(P_{n}\) is a positive local maximum point. 
By Lemma \ref{localmax}, we deduce that
\[
\frac{\lambda_{n}}{|u_{n}(P_{n})|^{p-2}}\to\tilde{\lambda}\in[0,1] \text{ as }n\to\infty.
\]
Next, we show that \(\tilde{\lambda}>0\).
Define the rescaled function
\[
\tilde{u}_{n}(x):=\tilde{\varepsilon}_{n}^{\frac{2}{p-2}}u_{n}(\tilde{\varepsilon}_{n}x +P_{n})\ \text{ for }\ x\in\tilde{\Omega}_{n}:=\frac{\Omega-P_{n}}{\tilde{\varepsilon}_{n}}.
\]
Clearly, \(\tilde{u}_{n}\) satisfies
\[
\left\{\begin{array}{lll}
-\Delta\tilde{u}_{n}+\lambda_{n}\tilde{\varepsilon}_{n}^{2}\tilde{u}_{n}=\rho_n|\tilde{\varepsilon}_{n}x +P_{n}|^{-\alpha}\tilde{u}_{n}^{p-1} & \text{ in }\tilde{\Omega}_n, \\
\tilde{u}_{n}(x)\leq\tilde{u}_{n}(0)=1 & \text{ in }\tilde{\Omega}_n\cap B_{R_n}(0), \\
\tilde{u}_{n}=0 & \text{ on }\partial\tilde{\Omega}_n,
\end{array}\right.
\]
Since $P_n\in \Omega$ is a point of local maximum of $u_n,$ we have 
\begin{equation*}
    0 \leq -\Delta \tilde{u}_n(0) = \rho_n|P_n|^{-\alpha} - \lambda_n \tilde{\varepsilon}_n^2  \implies \lambda_n \tilde{\varepsilon}_n^2  \leq \rho_n|P_n|^{-\alpha}\leq 1,
\end{equation*}
it follows that, up to a subsequence,
\begin{equation*}
    \rho_n|P_n|^{-\alpha}\rightarrow \Lambda_0 \quad \text{as} \quad n\rightarrow \infty,
\end{equation*}
for some $\Lambda_0\in[\tilde{\lambda},1].$
Let \(d_{n}=\text{dist}(P_{n},\partial\Omega)\). Then
\[
\frac{d_{n}}{\tilde{\varepsilon}_{n}}:=L\in[0,+\infty]\ \text{ and }\ \tilde{\Omega}_n\to \left\{\begin{array}{lll}
\mathbb{R}^{N} & \text{if}\ \ L=+\infty; \\
\mathbb{H} & \text{if}\ \ L<+\infty,
\end{array}\right.
\]
where \(\mathbb{H}\) denotes a half-space such that \(0\in\mathbb{H}\) and \(d(0,\partial\mathbb{H})=L\). By regularity arguments, up to a subsequence, \(\tilde{u}_{n}\to\tilde{u}\geq0\) in \(C^{2}_{loc}(\overline{D}),\) where \(\tilde{u}\) solves
\[
\left\{\begin{array}{lll}
-\Delta\tilde{u}+\tilde{\lambda}\tilde{u}=\Lambda_0\tilde{u}^{p-1} & \text{in }D, \\
\tilde{u}(x)\leq\tilde{u}(0)=1 & \text{ in }D, \\
\tilde{u}=0 & \text{ on }\partial D,
\end{array}\right.
\]
where \(D\) is either \(\mathbb{R}^{N}\) or \(\mathbb{H}\).

We claim that \(m(\tilde{u})\leq 2\). To see this, suppose for contradiction that there exists \(k>2\) such that there are \(k\) positive functions \(\phi_{1},\cdots,\phi_{k}\in H^{1}(D)\cap C_c(D),\) orthogonal in \(L^{2}(\Omega),\) satisfying
\[
\int_{D}|\nabla\phi_{i}|^{2}dx+\tilde{\lambda}\int_{D}\phi_{i}^{2}dx-\int_{D}(p-1)\Lambda_0\tilde{u}^{p-2}\phi_{i}^{2}dx<0
\]
for every \(i\in\{1,\cdots,k\}\). Because $\tilde{\Omega}_n\rightarrow D,$ for $n$ large enough, we have 
\[\bigcup_{i=1}^{k}supp (\phi_i)\subset\tilde{\Omega}_n,\quad \text{where supp$(\phi_i)$ is the support of $\phi_i$}.\]
Define the rescaled functions
\[
\phi_{i,n}(x):=\tilde{\varepsilon}_{n}^{-\frac{N-2}{2}}\phi_{i}\left(\frac{x-P_{n}}{\tilde{\varepsilon}_{n}}\right),
\]
so the support of $\phi_{i,n}$ belongs to $\Omega,$ thus $\phi_{i,n}\in H_0^1(\Omega)\cap C_c(\Omega),$ and we can verify that
\begin{align}
&\int_{\Omega}|\nabla\phi_{i,n}|^{2}dx+\lambda_n\int_{\Omega}\phi_{i,n}^{2}dx-\int_{\Omega}\rho_n(p-1)|x|^{-\alpha}|u_n^{p-2}\phi_{i,n}^{2}dx\\
&=\int_{\Omega_n}|\nabla\phi_{i}|^{2}dx+\lambda_n\tilde{\varepsilon}_n^2\int_{\Omega_n}\phi_{i}^{2}dx-\int_{\Omega_n}\rho_n(p-1)|\tilde{\varepsilon}_nx+P_n|^{-\alpha}|\tilde{u}_n^{p-2}\phi_{i}^{2}dx\\
&\rightarrow \int_{D}|\nabla\phi_{i}|^{2}dx+\tilde{\lambda}\int_{D}\phi_{i}^{2}dx-\int_{D}(p-1)\Lambda_0|\tilde{u}|^{p-2}\phi_{i}^{2}dx<0,
\end{align}
This implies that \(m(u_{n})\geq k>2\) for sufficiently large \(n,\) thereby yielding a contradiction. Thus, the claim is valid.
This allows us to conclude that the occurrence of \(\tilde{\lambda}=0\) is ruled out, regardless of whether \(D\) is a half-space or a whole space. Consequently, we can assert that \(\tilde{\lambda}\in(0,\infty)\).

In the sequel, we consider the sequence \((v_{n})\) defined by $($\ref{scaledsequence}$)$. Clearly, \(v_{n}\) satisfies
\[
\left\{\begin{array}{ll}
-\Delta v_{n}+v_{n}=\rho_n|\varepsilon_nx+P_n|^{-\alpha}v_n^{p-1} & \mbox{ in }\Omega_n, \\
v_{n}(x)\leq v_{n}(0)=(\frac{\varepsilon_n}{\tilde{\varepsilon}_n})^{\frac{2}{p-2}}\rightarrow\tilde{\lambda}^{-\frac{1}{p-2}} & \mbox{ in }\Omega_n\cap B_{R_n\frac{\tilde{\varepsilon}_n}{\varepsilon_n}}(0), \\
v_n=0 & \mbox{ on }\partial\Omega_n.
\end{array}\right.
\]
Up to a subsequence, we have \(v_{n}\to v\) in \(C^{2}_{\text{loc}}(D),\) where \(D\) is either \(\mathbb{R}^{N}\) or a half space \(\mathbb{H},\) and \(v\) solves
\[
\left\{\begin{array}{ll}
-\Delta v+v=\Lambda_0v^{p-1} & \mbox{ in }D, \\
v(x)\leq v(0)=\tilde{\lambda}^{-\frac{1}{p-2}} & \mbox{ in }D, \\
v=0 & \mbox{ on }\partial D.
\end{array}\right.
\]
And \(m(v)\leq 2\). 

By \cite[Theorem 2.5]{pierotti2017normalized}, we know \(D=\mathbb{R}^{N},\) then \(\limsup\limits_{n\to+\infty}\frac{\text{dist}(P_{n},\partial\Omega)}{\tilde{\varepsilon}_{n}}=+\infty.\) Using \(m(v)\leq 2,\) by standard regularity arguments, \(v\in C^{2}(\mathbb{R}^{N}),\) \(|v(x)|\rightarrow0\) as \(|x|\to+\infty\). Using the strong maximum principle we know that \(v>0,\) and it coincides with the unique radial ground-state solution \(W_0\) to \(-\Delta u+u=\Lambda_0u^{p-1}\) in \(\mathbb{R}^{N}\). 
All in all, we get that (ii) holds.

By \cite[Theorem 1.1]{EspositoPetralla2011}, there exists \( \phi \in C_0^\infty(\mathbb{R}^N) \) such that \(\text{supp } \phi \subset B_R(0),\) \( R > 0 ,\) and
\[
\int_{\mathbb{R}^N} |\nabla \phi|^2 +\tilde{\lambda} \phi^2 + (p-1)\Lambda_0W_0^{p-2} \phi^2 < 0.
\]
As before, the function \( \phi_n(x) := \tilde{\varepsilon}_n^{-\frac{N-2}{2}} \phi\left(\frac{x-P_n}{\tilde{\varepsilon}_n}\right) \) is what we are looking for in (iii). 

Moreover, in view of \( 2\frac{p}{p-2} - N > 0 \) and by Lemma \ref{localmax},
\[
\tilde{\varepsilon}_n^2 \leq \lambda_n^{-1} ,
\]
let \( n \to +\infty ,\) we have that
\begin{align}
\int_{B_R(0)} (W_0)^{q} = \lim_{n \to +\infty} \int_{B_R(0)} v_n^{q} &= \lim_{n \to +\infty} \varepsilon_n^{\frac{2q}{p-2}-N} \int_{B_{R \varepsilon_n (P_n)}} u_n^{q}\\
&= \lim\limits_{n\rightarrow+\infty} \lambda_n^{\frac{N}{2}-\frac{q}{p-2}} \int_{B_{R \varepsilon_n (P_n)}} u_n^{q},
\end{align}
and (iv) is also proved.
\end{proof}
\end{lemma}

We now proceed to provide a comprehensive global blow-up analysis. The following result offers a detailed description of the asymptotic behavior of \(\{u_{n}\}\) as \(\lambda_{n} \to +\infty\).

\begin{lemma}\label{Pn12k}
There exists \(k \in \{1, 2\}\) and sequences of points \(\{P^{1}_{n}\}, \cdots, \{P^{k}_{n}\},\) such that
\begin{equation}\label{maxPni}
|u_{n}(P^{i}_{n})| = \max_{B_{R_n\lambda_{n}^{-1/2}}(P^{i}_{n}) \cap \Omega} |u_{n}| \text{ for some } R_{n} \to \infty \text{, for every } i,
\end{equation}

\begin{equation}\label{lambdaPni}
\lambda_{n}|P^{i}_{n} - P^{j}_{n}|^{2} \to +\infty, \;\; \forall i \neq j, \;\; \text{as}\ \  n \to \infty,
\end{equation}
and moreover,
\begin{equation} \label{limitoflambdanun} 
\lim_{R \to \infty} \left( \limsup_{n \to +\infty} \lambda_{n}^{-\frac{1}{p-2}} \max_{d_{n} \geq R\lambda_{n}^{-1/2}} |u_{n}(x)| \right) = 0,
\end{equation}
where \(d_{n}(x) = \min\{|x - P^{i}_{n}| : i = 1, \cdots, k\}\) is the distance function from \(\{P^{1}_{n}, \cdots, P^{k}_{n}\}\) for \(x \in \Omega\).

\begin{proof} Take \(P^{1}_{n} \in \Omega\) such that \(u_{n}(P^{1}_{n}) = \max\limits_{\Omega} |u_{n}(x)|\). If $($\ref{limitoflambdanun}$)$ is satisfied for \(P^{1}_{n},\) then we get \(k = 1\). It is evident that \(P^{1}_{n}\) satisfies $($\ref{maxPni}$)$.

Otherwise, if \(P^{1}_{n}\) does not satisfy $($\ref{limitoflambdanun}$),$ we suppose that there exists \(\delta > 0\) such that
\[
\lim_{R \to \infty} \left( \limsup_{n \to +\infty} \lambda_{n}^{-\frac{1}{p-2}} \max_{|x - P^{1}_{n}| \geq R\lambda_{n}^{-1/2}} |u_{n}(x)| \right) \geq 4\delta.
\]
For sufficiently large \(R,\) up to a subsequence, it holds
\begin{equation}\label{Pn1}
\lambda_{n}^{-\frac{1}{p-2}} \max_{|x - P^{1}_{n}| \geq R\lambda_{n}^{-1/2}} |u_{n}(x)| \geq 2\delta.
\end{equation}

Let \(P^{2}_{n} \in \Omega \setminus B_{R\lambda_{n}^{-1/2}}(P^{1}_{n})\) such that
\[
|u_{n}(P^{2}_{n})| = \max_{\Omega \setminus B_{R\lambda_{n}^{-1/2}}(P^{1}_{n})} |u_{n}|.
\]
Then, $($\ref{Pn1}$)$ yields that \(|u_{n}(P^{2}_{n})|\geq 2\delta\lambda_n^{\frac{1}{p-2}} \to +\infty\) as \(n \to +\infty\).
We claim that
\begin{equation}\label{Pn12}
\lambda_{n}|P^{1}_{n} - P^{2}_{n}|^{2} \to +\infty.
\end{equation}
If $($\ref{Pn12}$)$ is not true, then up to subsequence \(\lambda^{\frac{1}{2}}_{n}|P^{1}_{n} - P^{2}_{n}| \to R^{\prime} \geq R\). Define
\[
v_{n,1}(x) :=  \lambda_{n}^{-\frac{1}{p-2}} u_{n}(\lambda_{n}^{-1/2}x + P^{1}_{n}).
\]

As in Lemma \ref{complication}, we can deduce that \(v_{n,1} \to v\) in \(C^{2}_{loc}(D),\) where \(D = \mathbb{R}^{N}\). Then, up to subsequences,
\[
\lambda_{n}^{-\frac{1}{p-2}} u_{n}(P^{2}_{n}) = v_{n,1}({\lambda_n^{1/2}(P^{2}_{n} - P^{1}_{n})}) \to v(x'), \; |x^{\prime}| \geq R^{\prime} > R.
\]
Since \(v(x) \to 0\) as \(|x| \to \infty,\) taking \(R\) larger if necessary, it follows that
\[
|v(x)| \leq  \delta, \forall |x| \geq R.
\]
This contradicts to $($\ref{Pn1}$),$ which proves the claim $($\ref{Pn12}$)$.

In the following, we shall show that
\begin{equation}\label{Pn2}
|u_{n}(P^{2}_{n})| = \max_{\Omega \cap B_{R_{n,2}\lambda^{-1/2}}(P^{2}_{n})} |u_{n}(x)|, \text{ for some } R_{n,2} \to +\infty.
\end{equation}
Let \(\tilde{\varepsilon}_{n,2} = |u_{n}(P^{2}_{n})|^{-\frac{p-2}{2}}\). Clearly, by lemma \ref{localmax}, \(\tilde{\varepsilon}_{n,2} \to 0\). By $($\ref{Pn1}$)$ we get \(\tilde{u}_{n,2} \leq (2\delta)^{-\frac{p-2}{2}} \lambda_{n}^{-\frac{1}{2}}\). From $($\ref{Pn12}$),$ we can assert that
\[
\tilde{R}_{n,2} := \frac{|P^{1}_{n} - P^{2}_{n}|}{2\tilde{\varepsilon}_{n,2}} \geq \frac{(2\delta)^{\frac{p-2}{2}}}{2} \lambda_{n}^{\frac{1}{2}} |P^{1}_{n} - P^{2}_{n}| \to +\infty, \text{ as }\  n \to \infty.
\]
On the other hand, for any \(x \in B_{\tilde{R}_{n,2},\tilde{\varepsilon}_{n,2}}(P^{2}_{n})\) and \(R > 0,\) we have
\[
|x - P^{1}_{n}| \geq |P^{2}_{n} - P^{1}_{n}| - |x - P^{2}_{n}| \geq \frac{1}{2} |P^{2}_{n} - P^{1}_{n}| \geq R \lambda_{n}^{-\frac{1}{2}}
\]
for arbitrarily large \(n\). Consequently,
\[
\Omega \cap B_{\tilde{R}_{n,2},\tilde{\varepsilon}_{n,2}}(P^{2}_{n}) \subset \Omega \backslash B_{R\lambda_n^{-\frac{1}{2}}} \{P^{1}_{n}\},
\]
which implies that
\[
|u_{n}(P^{2}_{n})| = \max_{\Omega \cap B_{\tilde{R}_{n,2},\tilde{\varepsilon}_{n,2}}(P^{2}_{n})} |u_{n}|.
\]

We define
\[
\tilde{u}_{n,2}(x) := \tilde{\varepsilon}_{k_{n,2}}^{\frac{2}{p-2}} u_{n}(\tilde{\varepsilon}_{n,2}x + P^{2}_{n}), \quad x \in \tilde{\Omega}_{n,2} := \frac{\Omega - P^{2}_{n}}{\tilde{\varepsilon}_{n,2}}.
\]
Set \(d_{n,2} = \operatorname*{dist}(P^{2}_{n}, \partial\Omega)\). Then, passing to subsequences if necessary,
\[
\frac{d_{n,2}}{\tilde{\varepsilon}_{n,2}} \to L_{2} \in [0, +\infty] \ \text{ and } \ \tilde{\Omega}_{n,2} \to 
\begin{cases}
\mathbb{R}^{N} & \text{if } L_{2} = +\infty, \\
\mathbb{H} & \text{if } L_{2} < +\infty.
\end{cases}
\]
Then \(\tilde{u}_{n,2}\) satisfies the following equation
\[
\left\{
\begin{array}{ll}
-\Delta\tilde{u}_{n,2} + \lambda_{n}\tilde{\varepsilon}_{n,2}^{2}\tilde{u}_{n,2} =\rho_n|\tilde{\varepsilon}_{n,2}x + P^{2}_{n}|^{-\alpha} \tilde{u}_{n,2}^{p-1} & \text{in } \tilde{\Omega}_{n,2}, \\
\tilde{u}_{n,2}(x) \leq \tilde{u}_{n,2}(0) = 1 & \text{in } \tilde{\Omega}_{n,2}\cap B_{\tilde{R}_{n,2}}(0), \\
\tilde{u}_{n,2} = 0 & \text{on } \partial\tilde{\Omega}_{n,2}.
\end{array}
\right.
\]
Since \(P^{2}_{n}\) is a local maximum, by Lemma \ref{localmax} we get
\[
\frac{\lambda_{n}}{|u_{n}(P^{2}_{n})|^{p-2}} \to \tilde{\lambda}^{(2)} \in [0, 1], \text{ as } n \to \infty.
\]
Using the similar argument as in Lemma \ref{complication}, we deduce \(\tilde{\lambda}^{(2)} > 0,\) namely
\[
\lim_{n \to +\infty} \lambda_{n}^{\frac{1}{2}} \tilde{\varepsilon}_{n,2} > 0.
\]
Set
\[
v_{n,2} := \varepsilon_{n}^{\frac{2}{p-2}} u_{n}({\varepsilon}_{n}x + P_{n}^{2})\ \  \text{ for }\  x \in \Omega_{n,2} := \frac{\Omega - P_{n}^{2}}{{\varepsilon}_{n}},
\]
where \({\varepsilon}_{n} = \lambda_{n}^{-\frac{1}{2}}\) is defined in Lemma \ref{complication}. Since $P_n^2\in \Omega$ is a point of local maximum of $u_n,$ we have 
\begin{equation*}
    0 \leq -\Delta v_{n,2}(0) = \rho_n|P_n^2|^{-\alpha} - \lambda_n \varepsilon_n^2  \implies \lambda_n \varepsilon_n^2  \leq \rho_n|P_n^2|^{-\alpha}\leq 1,
\end{equation*}
it follows that, up to a subsequence,
\begin{equation*}
    \rho_n|P_n^2|^{-\alpha}\rightarrow \Lambda_2 \quad \text{as} \quad n\rightarrow \infty,
\end{equation*}
for some $\Lambda_2\in[\tilde{\lambda}^{(2)},1].$Up to a subsequence, there exists a function \(v^{(2)} \in H^{1}(D)\) such that \(v_{n,2} \to v^{(2)}\) in \(C^{2}_{loc}(D),\) where \(D\) is \(\mathbb{R}^{N}\). Moreover, \(v^{(2)}\) solves the following problem
\[
\left\{
\begin{array}{ll}
-\Delta v^{(2)} +  v^{(2)} = \Lambda_2(v^{(2)})^{p-1} & \text{in } D, \\
v^{(2)}(x) \leq v^{(2)}(0) = (\tilde{\lambda}^{(2)})^{-\frac{1}{p-2}} & \text{in } D, \\
v = 0 & \text{on } \partial D.
\end{array}
\right.
\]
Then, by a similar discussion as Lemma \ref{complication}, we conclude that \(\frac{d_{n,2}}{\varepsilon_n} \to +\infty\). Define \(R_{n,2} = \tilde{R}_{n,2} \lambda^{\frac{1}{2}}_{n} \tilde{\varepsilon}_{n,2}\). Clearly, \(R_{n,2} \to +\infty\) as \(n \to +\infty\). Hence, $($\ref{Pn2}$)$ holds.

If $($\ref{limitoflambdanun}$)$ does not hold, we can apply similar arguments as before to show that there exists \(P_{n}^{3}\) such that $($\ref{maxPni}$)$-$($\ref{lambdaPni}$)$ are satisfied. For \(P^{i}_{n}, i = 1, 2, 3,\) applying Lemma \ref{complication} again, we can find \(\phi^{i}_{n} \in C_{0}^{\infty}(\Omega)\) with supp\(\phi^{i}_{n} \in B_{R\varepsilon_{n}}(P_{n}^{i}) \cap \Omega\) for some \(R > 0,\) such that
\[
\int_{\Omega} [\nabla\phi^{i}_{n}]^{2} \, dx + \int_{\Omega} [(\lambda_n - (p - 1)\rho_{n}|x|^{-\alpha}u_{n}^{p-2})(\phi^{i}_{n})^{2}] \, dx < 0.
\]

In light of $($\ref{lambdaPni}$),$ we observe that \(\phi^{1}_{n}, \phi^{2}_{n}\) \(\phi^{i}_{n}\) are mutually orthogonal for sufficiently large \(n,\) which implies \(\lim\limits_{n \to +\infty} m(u_{n}) \geq 3\). This leads to a contradiction with the fact that \(m(u_{n}) \leq 2\). 
 \end{proof}
In the subsequent analysis, we show that \(u_{n}\) exhibits exponential decay away from the blow-up points.
   
\end{lemma}
\begin{lemma}\label{estimate} Let \(\{P_{n}^{1}, \cdots, P_{n}^{k}\}\) be given in Lemma \ref{Pn12k}, Then there are \(\gamma,C>0\) so that
\begin{equation}\label{estimate'}
u_{n}(x)\leq C\,\lambda_{n}^{\frac{1}{p-2}}\sum_{i=1}^{k}e^{-\gamma\lambda_{n}^{\frac{1}{2}}|x-P_N^i|}\quad\forall x\in\Omega,\quad n\in\mathbb{N}.
\end{equation}
\begin{proof}
By $($\ref{limitoflambdanun}$)$ for \(R>0\) large and \(n\geq n(R)\) there holds
\[
\lambda_{n}^{-\frac{1}{p-2}}\max_{\left\{d_{n}(x)\geq R\lambda_{n}^{- \frac{1}{2}}\right\}}u_{n}(x)\leq\left(\frac{1}{2}\right)^{\frac {1}{p-2}},
\]
hence, in \(\{d_{n}(x)\geq R\,\lambda_{n}^{-\frac{1}{2}}\}\) for \(n\geq n(R)\) we have
\begin{equation}\label{5.15}
\tilde{a}_{n}(x):=\lambda_{n}-\rho_n|x|^{-\alpha}u_{n}^{p-2}(x)\geq\frac{\lambda_{n }}{4}. 
\end{equation}

Compute the linear operator \(-\Delta+\tilde{a}_{n}(x)\) on \(\xi_{n}^{i}(x)=e^{-\gamma\lambda_{n}^{\frac{1}{2}}|x-P_{n}^{i}|}\) in \(\{d_{n}(x)\geq R\,\lambda_{n}^{-\frac{1}{2}}\}\):
\[
(-\Delta+\tilde{a}_{n})(\xi_{n}^{i})=\lambda_{n}\xi_{n}^{i}\bigg{[}-\gamma^{2}+( N-1)\frac{\gamma}{\lambda_{n}^{\frac{1}{2}}|x-P_{n}^{i}|}+\lambda_{n}^{-1} \tilde{a}_{n}(x)\bigg{]}\geq 0
\]
for \(n\) large, provided $\gamma$ small enough. Observe that for \(R\) large
\begin{equation}\label{xilimit}
\bigg{(}e^{\gamma R}\xi_{n}^{i}(x)-\lambda_{n}^{-\frac{1}{p-2}}u_{n}(x)\bigg{)} \Big{|}_{\partial B_{R\,\lambda_{n}^{-\frac{1}{2}}}(P_{n}^{i})}\to 1-W_0(R)>0
\end{equation}
as \(n\to+\infty,\) where $W_0(|x|)=W_0(x)$ for any $x\in\mathbb{R}^N$.

Then, if we define \(\xi_{n}:=e^{\gamma\,R\,}\lambda_{n}^{\frac{1}{p-2}}\sum_{i=1}^{k}\xi_{n}^{i},\) for \(L_{n}=-\Delta+\lambda_{n}-\rho_n|x|^{-\alpha}u_{n}^{p-2}\) we have
\[
L_{n}(\xi_{n}-u_{n})=L_n\xi_n=e^{\gamma\,R\,}\lambda_{n}^{\frac{1}{p-2}}\sum_{i=1}^{k}(-\Delta+\tilde{a}_{n})\xi_{n}^{i}\geq 0\quad\text{in $\{d_{n}(x)>R\,\lambda_{n}^{-\frac{1}{2 }}\}$}
\]
and by $($\ref{xilimit}$),$ \(\xi_{n}-u_{n}\geq 0\) on \(\{d_{n}(x)=R\,\lambda_{n}^{-\frac{1}{2}}\}\cup\partial\Omega\). Note that by Lemma \ref{Pn12k}
\[
\left\{d_{n}(x)=R\,\lambda_{n}^{-\frac{1}{2}}\right\}=\bigcup_{i=1}^{k}\partial B_{ R\,\lambda_{n}^{-\frac{1}{2}}}(P^{i}_{n})\subset\Omega
\]
for \(n\geq n(R)\). Then, by $($\ref{5.15}$)$ and the maximum principle of E.Hopf \cite[Theorem 6 in Section 3]{Maximum}
\[
u_{n}\leq\xi_{n}=e^{\gamma\,R}\,\lambda_{n}^{\frac{1}{p-2}}\sum_{i=1}^{k}e^{-\gamma \lambda_{n}^{\frac{1}{2}}\,|x-P^{i}_{n}|}
\]
in \(\{d_{n}(x)\geq R\,\lambda_{n}^{-\frac{1}{2}}\},\) if \(R\) is large and \(n\geq n(R)\). Since by the definition in Lemma \ref{complication} and Lemma \ref{Pn12k},
\[
u_{n}(x)\leq\max_{\Omega}u_{n}=(\tilde{\varepsilon}_{n})^{-\frac{2}{p-2}}\,\leq\,Ce^{ \gamma\,R}\,\lambda_{n}^{\frac{1}{p-2}}\sum_{i=1}^{k}e^{-\gamma\lambda_{n}^{\frac{1}{2}}\,| x-P^{i}_{n}|}
\]
for some \(C>0,\) if \(d_{n}(x)\leq R\,\lambda_{n}^{-\frac{1}{2}},\) we have that $($\ref{estimate'}$)$ holds true in \(\Omega\) with a constant \(Ce^{r\,R}\) and \(n\geq n(R)\). Up to take a larger constant \(C,\) we have the validity of $($\ref{estimate'}$)$ in \(\Omega\) for every \(n\in\mathbb{N}\).
\end{proof}
\end{lemma}
Building on the blow-up analysis above, we shall prove the following results.
\begin{proof}[Proof of Theorem \ref{submass}]
By Proposition \ref{prop3.5}, for any $\rho\in[\frac{1}{2},1],$ there is a sequence $\{\rho_n\}\subset[\frac{1}{2},1]$ such that $\rho_n\to \rho$ and $u_{\rho_n} $ is the positive $(\mu,\rho_n)$ sub-mass mountain pass solution for $\lambda_{\rho_n}\geq0$. We will prove the case where $\rho=1,$ writing for simplicity $u_n:=u_{\rho_n}$ for $\rho_n\to1^-$; the proofs for other cases are similar and left to interested readers.
    Firstly, we show that \(\{\lambda_{n}\}\) is bounded. Suppose, by contradiction, that \(\lambda_{n}\rightarrow+\infty\). By Lemma \ref{Pn12k}, there exist at most \(k\) blow-up limits \(\{P_{n}^{1}\},\cdots,\{P_{n}^{k}\}\) with \(k\leq 2\). In the following, we denote by \(\{v_{n}^{l}:=  \lambda_{n}^{-\frac{1}{p-2}} u_{n}(\lambda_{n}^{-1/2}x + P^{l}_{n})\}\) the scaled sequence around \(\{P_{n}^{l}\},\) and by Lemma \ref{complication} we know\(W_0\) is the limit of \(v_{n}^{l},\) for all $l=1,\cdots,k$. Note that for these blow-up points \(\{P_{n}^{1}\},\cdots,\{P_{n}^{k}\},\) by Lemma \ref{complication}  it hold
\[
\lambda_{n}^{\frac{1}{2}}\text{dist}(P_{n}^{l},\partial\Omega)\rightarrow+\infty,\quad \forall l=1,\cdots,k.
\]
On the one hand, we can deduce that for any \(R>0,\)
\begin{equation}\label{tag5.1}
\left|\lambda_{n}^{\frac{N}{2}-\frac{2}{p-2}}\int_{\Omega}u_{n}^{2}dx-\sum_{l=1}^{k}\int_{B_{R}(0)}|v_{n}^{l}|^{2}dx\right|\rightarrow+\infty. 
\end{equation}
In fact, from the proof of Proposition \ref{prop3.5}, the $(\mu,\rho_n)$ sub-mass mountain pass sequence $\{u_n\}$ is a subset of $S_\mu$ when $\lambda_n\to\infty,$ then since \(p\in(2+4/N,2^{*}),\) the first term satisfies
\[
\left|\lambda_{n}^{\frac{N}{2}-\frac{2}{p-2}}\int_{\Omega}u_{n}^{2}dx\right|=\lambda_{n}^{\frac{N}{2}-\frac{2}{p-2}}\mu\rightarrow+\infty.
\]
By Lemma Lemma \ref{complication}, we have
\[
\sum_{i=1}^{k}\int_{B_{R}(0)}|v_{n}^{i}|^{2}dx\rightarrow\sum_{i=1}^{k}\int_{B_{R}(0)}|W_0|^{2}dx,
\]
which imply that $($\ref{tag5.1}$)$ holds.

On the other hand, by Lemma \ref{estimate}, there exist constants \(C,C^{\prime}>0\) such that
\begin{align*}
&\left|\lambda_{n}^{\frac{N}{2}-\frac{2}{p-2}}\int_{\Omega}u_{n}^{2}dx-\sum_{l=1}^{k}\int_{B_{R}(0)}|v_{n}^{l}|^{2}dx\right| \\
&=\lambda_{n}^{\frac{N}{2}-\frac{2}{p-2}}\int_{\Omega\setminus\cup_{l=1}^{k}(B_{R \lambda_{n}^{-1/2}}(\mathbb{P}_{N}^{l}))}u_{n}^{2}dx \\
&\leq C \lambda_{n}^{\frac{N}{2}}\sum_{l=1}^{k}\int_{\mathbb{R}^{N}\setminus\cup_{l=1}^{k}(B_{R \lambda_{n}^{-1/2}}(\mathbb{P}_{N}^{l}))}e^{-2\gamma\lambda_{n}^{\frac{1}{2}}|x-p_n^l|}dx \\
&\leq C \sum_{i=1}^{k}\int_{\mathbb{R}^{N}\setminus B_{R}(0)}e^{-2\gamma|y|}dy\leq C^{\prime}=C^{\prime}(R).
\end{align*}

Taking \(n\to+\infty,\) we obtain a contradiction to $($\ref{tag5.1}$)$. Hence, \(\{\lambda_{n}\}\) is bounded, so $\{u_n\}$ is bounded too, and up to a subsequence, there is a $u(1)\in H_0^1(\Omega)$ such that $u_n\rightharpoonup u(1)$ as $n\to \infty.$ Similar to the proof process of Proposition \ref{prop3.5}, because $c_\rho(\mu)$ is continuous with respect to $\rho,$ $u(1)$ is a positive $(\mu,1)$ sub-mass mountain pass solution. 

\end{proof}

\begin{proof}[Proof of Theorem \ref{asymtotic}]
   Let $(u(\mu,\rho),\lambda(\mu,\rho))$ be given by Theorem \ref{submass}. By Theorem \ref{c_rho} and the definition of $c_\rho(\mu)$ we know that 
\[c_\rho(\mu)\geq \frac{k_0}{4} =  \frac{1}{4}\left(\frac{ p}{4C_{p,N}}\right)^{\frac{2}{p\gamma_{p}-2}} \mu^{-\frac{(1-\gamma_{p})p}{p\gamma_{p}-2}}\to \infty\quad \text{as}\quad \mu\to 0,\]
because $-\frac{(1-\gamma_{p})p}{p\gamma_{p}-2}<0.$ At the other hand, $(u(\mu,\rho),\lambda(\mu,\rho))$ solves equation $($\ref{1.1rho}$)$ and satisfies that $c_\rho(u(\mu,\rho))=E_\rho(u(\mu,\rho)),$ so 
\begin{align*}c_\rho(u(\mu,\rho))&=\frac{p-2}{2p}\int_{\Omega}|x|^{-\alpha}|u(\mu,\rho)|^{p}-\frac{\lambda(\mu,\rho)}{2}\int_{\Omega}|u(\mu,\rho)|^2\\
&\leq \frac{p-2}{2p}\max_{\Omega}|u(\mu,\rho)(x)|^{p-2}r_0^{-\alpha}\mu,\end{align*}
 and then \begin{align}\label{blowupofmu} \tilde{\varepsilon}_\mu:=\max_{\Omega}|u(\mu,\rho)(x)|\to\infty,\quad \text{as}\ \mu\to 0.\end{align}
 To prove the final conclusion, we will divide the proof process into two steps.

 Step 1: We claim that there is a $\mu_0$ small enough, such that $\lambda(\mu,\rho)>0$ for all $0<\mu<\mu_0.$ 
 
 If not, we can find a subsequence $\{(\lambda_k,u_k):=(u(\mu_k,\rho),\lambda(\mu_k,\rho))\}$ with $\lambda_k\equiv 0$ and $\mu_k\to0$ as $k\to\infty$. Define the rescaled function
\[
\tilde{u}_k(x) := \tilde{\varepsilon}_{k}^{\frac{2}{p-2}} u_k(\tilde{\varepsilon}_{k} x + P_{k}) \quad\text{ for }\  x \in \tilde{\Omega}_{k} := \frac{\Omega - P_{k}}{\tilde{\varepsilon}_{k}},
\]
where $P_k$ is the maximum point of $u_k$ and $\tilde{\varepsilon}_k=\tilde{\varepsilon}_{\mu_k}.$
Clearly, $\tilde{u}_k$ satisfies
\[
\begin{cases}
-\Delta \tilde{u}_k= |\tilde{\varepsilon}_{k} x + P_{k}|^{-\alpha} \tilde{u}_k^{p-1} & \text{in } \tilde{\Omega}_{k}, \\
\tilde{u}_k(x) \leq \tilde{u}_k(0) = 1 & \text{in } \tilde{\Omega}_{k}, \\
\tilde{u}_k = 0 & \text{on } \partial \tilde{\Omega}_{k},
\end{cases}
\]
Since $P_{k} \in \Omega$ is the maximum point $u_k,$ we have
\[
0 \leq -\Delta \tilde{u}_k(0) =  |P_{k}|^{-\alpha}  \implies 0 \leq  |P_{k}|^{-\alpha} \leq r_0^{-\alpha},
\]
it follows that, up to a subsequence,
\[
 |P_{k}|^{-\alpha} \to \Lambda_{0} \quad \text{as} \quad k \to 0,
\]
for some $\Lambda_{0} \in [0, 1]$. Let $d_{k} = \operatorname{dist}(P_{k}, \partial \Omega)$. Then
\[
\frac{d_{k}}{\tilde{\varepsilon}_{k}} := L \in [0, +\infty] \text{ and } \tilde{\Omega}_{k} \to D:= \begin{cases}
\mathbb{R}^{N} & \text{if } L = +\infty; \\
\mathbb{H} & \text{if } L < +\infty,
\end{cases}
\]
where $\mathbb{H}$ denotes a half-space such that $0 \in \mathbb{H}$ and $d(0, \partial \mathbb{H}) = L$. By regularity arguments, up to a subsequence, $\tilde{u}_k \to \tilde{u} \geq 0$ in $C^{2}_{loc}(\overline{D}),$ and $\tilde{u}$ solves
\[
\begin{cases}
-\Delta \tilde{u} = \Lambda_{0} \tilde{u}^{p-1} & \text{in } D, \\
\tilde{u}(x) \leq \tilde{u}(0) = 1 & \text{in } D, \\
\tilde{u} = 0 & \text{on } \partial D,
\end{cases}
\]
where $D$ is either $\mathbb{R}^{N}$ or $\mathbb{H}$.
By the Liouville theorem, we deduce $\Lambda_0=0,$ hence $\Delta u_k\equiv0,$ and $|P_k|\to \infty,$ then the strong maximum principle implies that $u_k\equiv1$ and so $D$ must be $\mathbb{H}$ and $L=0,$ but we can calculate directly that $d_k=\infty$ and $L=\infty$ from the definition of $d_k,$ $|P_k|\to \infty$ and $($\ref{blowupofmu}$)$. A contradiction!

Step 2: $ \lambda(\mu,\rho)\to\infty,$ when $\mu<\mu_0$ and $\mu\to 0.$\\
Just repeat the blow-up analysis similar to Step 1.

 Apply Theorem \ref{localmax} and Theorem \ref{complication}, there are $P\in\Omega$ and  
\[v(\mu,\rho):=\lambda(\mu,\rho)^{-\frac{1}{p-2}} u(\mu,\rho)(\frac{x}{\sqrt{\lambda(\mu,\rho)}} + P) \quad\text{ for }\  x \in \tilde{\Omega}_{\mu} := \frac{\Omega - P}{\sqrt{\lambda(\mu,\rho)}}\]
such that $v(\mu,\rho)\to V_\rho(x)$ in $C_{loc}^{2}(\mathbb{R}^N)$ as $\mu\to 0,$
where \(V_\rho\) is the unique positive solution of
\[\left\{\begin{array}[]{ll}-\Delta V_\rho+V_\rho=\rho\Lambda_0V_\rho^{p-1}&\quad\text{ in }\mathbb{R}^{N}, \\ V_\rho(0)=\max\limits_{x\in\mathbb{R}^N}V_\rho,&\\ V_\rho(x)\to 0~{}~{}\text{as}~{}~{}|x|\to+\infty&\end{array}\right.\]
where $\Lambda_0=|P|^{-\alpha}>0.$
\end{proof}

To complete the proof of Theorem \ref{main result}, we only need to prove the following proposition.

\begin{proposition}\label{mainproposition}
    Let $\alpha,\mu>0,$ $N\geq3,$ $2+4/N<p<2^*,$ $\rho\in[\frac{1}{2},1],$ and $\Omega\subset\mathbb{R}^N$ be an exterior domain with $0\notin\overline\Omega$. Then there is a $\mu_0>0$ such that for any $\mu\in(0,\mu_0),$ equations $($\ref{1.2}$)$-$($\ref{1.1rho}$)$ has a positive solution for some $\lambda> 0.$
\end{proposition}

\begin{proof}
In Theorem \ref{submass}, we have shown that for any $\mu>0,$ $\rho\in[\frac{1}{2},1],$ there is a positive $(\mu,\rho)$ sub-mass mountain pass solution $u(\mu,\rho)$ of equation $($\ref{1.1rho}$)$. Keep $\rho$ constant, from Theorem \ref{asymtotic}, there is a $\mu_0>0$ such that $\lambda(\mu,\rho)>0$ for all $\mu\in(0,\mu_0)$. So we only need to prove $u(\mu,\rho)\in S_\mu$ for all $\mu\in(0,\mu_0).$ From Proposition \ref{prop3.5}, 
there is a sequence $\{u(\mu,\bar\rho_n)\}$ satisfying $\bar\rho_n\to \rho,$ and for every $\bar\rho_n,$ there is a bounded $PS$ sequence $\{u_{\bar\rho_n}^m\}_m\subset S_\mu$  satisfying $(i)-(iv)$ from Theorem \ref{monotonous} such that $u_{\bar\rho_n}^m\rightharpoonup u(\mu,\bar\rho_n)$ as $m\to \infty$. 
As in the proof of Proposition \ref{prop3.5}, for any $\mu\in(0,\mu_0),$ $\lambda(\mu,\rho)>0$ implies that $u_{\bar\rho_n}^m$ converges strongly to $u(\mu,\bar\rho_n)$ as $m\to\infty.$ Hence $u(\mu,\bar\rho_n)\in S_\mu.$ Repeat this process, we shall prove that $u(\mu,\bar\rho_n)$ converges strongly to $u(\mu,\rho)$ as $n\to\infty.$ Proof completed.
 \end{proof}

\begin{proof}[Proof of Theorem \ref{main result}]
    The result is obtained by setting $\rho=1$ in Proposition \ref{mainproposition}.
\end{proof}

\section*{Acknowledgments}
The research of Xiaojun Chang is partially supported by NSFC (12471102), NSF of Jilin Province (20250102004JC), and the Research Project of the Education Department of Jilin Province (JJKH20250296KJ).

\end{document}